\documentclass[12pt]{amsart}
\usepackage[english]{babel}
\usepackage[utf8]{inputenc}
\usepackage{amsmath, amssymb, amsthm, amsfonts, amsxtra, stmaryrd, enumerate, graphicx}
\usepackage[colorinlistoftodos]{todonotes}
\usepackage{placeins}

\newtheorem{theorem}{Theorem}[section]
\newtheorem{corollary}[theorem]{Corollary}
\newtheorem{claim}[theorem]{Claim}

\newtheorem{fact}[theorem]{Fact}
\newtheorem{proposition}[theorem]{Proposition}

\newtheorem{remark}[theorem]{Remark}
\usepackage{tikz-cd}
\newtheorem{definition}[theorem]{Definition}
\newtheorem{problem}[theorem]{Problem}

\theoremstyle{definition}
\newtheorem{example}[theorem]{Example}

\def\Ind#1#2{#1\setbox0=\hbox{$#1x$}\kern\wd0\hbox to 0pt{\hss$#1\mid$\hss}
\lower.9\ht0\hbox to 0pt{\hss$#1\smile$\hss}\kern\wd0}

\def\ind{\mathop{\mathpalette\Ind{}}}

\def\notind#1#2{#1\setbox0=\hbox{$#1x$}\kern\wd0
\hbox to 0pt{\mathchardef\nn=12854\hss$#1\nn$\kern1.4\wd0\hss}
\hbox to 0pt{\hss$#1\mid$\hss}\lower.9\ht0 \hbox to 0pt{\hss$#1\smile$\hss}\kern\wd0}

\def\nind{\mathop{\mathpalette\notind{}}}

\usepackage{etoolbox}
\patchcmd{\subsection}{-.5em}{.5em}{}{}

\title{Conant-Independence and Generalized Free Amalgamation}
\author{Scott Mutchnik}

\begin{document}

\begin{abstract}
We initiate the study of a generalization of Kim-independence, \textit{Conant-independence}, based on the notion of \textit{strong Kim-dividing} of Kaplan, Ramsey and Shelah. A version of Conant-independence  was originally introduced to prove that all $\mathrm{NSOP}_{2}$ theories are $\mathrm{NSOP}_{1}$. We introduce an axiom on stationary independence relations, essentially generalizing the ``freedom" axiom in some of the \textit{free amalgamation theories} of Conant, and show that this axiom provides the correct setting for carrying out arguments of Chernikov, Kaplan and Ramsey on $\mathrm{NSOP}_{1}$ theories relative to a stationary independence relation. Generalizing Conant's results on free amalgamation to the limits of our knowledge of the $\mathrm{NSOP}_{n}$ hierarchy, we show using methods from Conant as well as our previous work that any theory where the equivalent conditions of this local variant of $\mathrm{NSOP}_{1}$ holds is either $\mathrm{NSOP}_{1}$ or $\mathrm{SOP}_{3}$ and is either simple or $\mathrm{TP}_{2}$. We observe that these theories give an interesting class of examples of theories where Conant-independence is symmetric. This includes all of Conant's examples, the small cycle-free random graphs of Shelah, and the (finite-language) $\omega$-categorical Hrushovski constructions of Evans and Wong.

We then answer Conant's question on the existence of non-modular free amalgamation theories. We show that the generic functional structures of Kruckman and Ramsey are examples of non-modular free amalgamation theories. We also show that any free amalgamation theory is $\mathrm{NSOP}_{1}$ or $\mathrm{SOP}_{3}$, while an $\mathrm{NSOP}_{1}$ free amalgamation theory is simple if and only if it is modular.

Finally, we show that every theory where Conant-independence is symmetric is $\mathrm{NSOP}_{4}$. Therefore, symmetry for Conant-independence gives the next known neostability-theoretic dividing line on the $\mathrm{NSOP}_{n}$ hierarchy beyond $\mathrm{NSOP}_{1}$. We explain the connection to some established open questions.
\end{abstract}
\maketitle

\tableofcontents

\newpage

\section{Introduction}
One of the most rapidly evolving areas of model theory is the study of potentially non-$\mathrm{NSOP}_{1}$ $\mathrm{NSOP}$ theories. Two cornerstone problems of this field include determining the status of the open regions of this part of the classification-theoretic map, and developing a theory of independence for these theories\footnote{For a somewhat different tradition in the theory of independence for potentially non-simple theories, with some overlap with the higher $\mathrm{NSOP}_{n}$ hierarchy including the modular free amalgamation theories from \cite{Co15}, see \cite{Ons06}.}. One of the main questions of the first program, first formally posed by Džamonja and Shelah \cite{DS04} but originally asked by Shelah in notes in \cite{Sh99} based on lectures given at Rutgers University in fall on 1997, was whether the class $\mathrm{NSOP}_{2}$ coincides with the class $\mathrm{NSOP}_{1}$. This was recently answered in the affirmative by the author in \cite{NSOP2}. Yet the following question from \cite{DS04}, \cite{Sh99} remains open:

\begin{problem}
Is every $\mathrm{NSOP}_{3}$ theory $\mathrm{NSOP}_{2}$ (and therefore $\mathrm{NSOP}_{1}$?).
\end{problem}

An additional open question (\cite{FD}, \cite{Che14}),  involves the interactions of the $\mathrm{NSOP}_{n}$ hierarchy with $\mathrm{NTP}_{2}$:

\begin{problem}
Is the $\mathrm{NSOP}_{n}$ hierarchy strict within $\mathrm{NTP}_{2}$ (including $\mathrm{NSOP}_{n}$ for $n \geq 3$ as well as $\mathrm{NSOP}$ itself?) That is, are the inclusions

$$\mathrm{simple} \subseteq \mathrm{NTP}_{2} \cap \mathrm{NSOP}_{3} \subseteq \ldots \subseteq \mathrm{NTP}_{2} \cap \mathrm{NSOP}_{n} \subseteq \ldots \mathrm{NTP}_{2} \cap \mathrm{NSOP}$$ 

strict?
\end{problem}

Note that Shelah (\cite{Sh90}) showed that all $\mathrm{NSOP}_{2}$ $\mathrm{NTP}_{2}$ theories are simple. Partial results on these problems include work of Evans and Wong in \cite{EW09} proving the $\omega$-categorical Hrushovski constructions introduced in \cite{Ev02} are either simple or strictly $\mathrm{NSOP}_{4}$, work of Conant in \cite{Co15} proving modular \textit{free amalgamation theories} are either simple or strictly $\mathrm{NSOP}_{4}$ $\mathrm{TP}_{2}$, work of Kaplan, Ramsey and Simon (\cite{CKR23}) shows that all binary theories are either $\mathrm{SOP}_{3}$ or $\mathrm{NSOP}_{2}$, and either $\mathrm{SOP}_{1}$ or simple (and therefore, by \cite{NSOP2}, either $\mathrm{SOP}_{3}$ or simple.) Yet none of the previous literature explicitly treats general classes of theories that approach the limits of our understanding of the $\mathrm{NSOP}_{n}$ hierarchy: potentially $\mathrm{SOP}_{3}$, but also potentially \textit{strictly} $\mathrm{NSOP}_{1}$. (However, work of Johnson and Ye, introducing \textit{curve-excluding fields} (\cite{JY23}) known to be $\mathrm{TP}_{2}$ and thus not simple but thought to be $\mathrm{NSOP}_{4}$, deserves mention; see below.) One of the goals of this paper is to show that the potentially \textit{non-modular} free amalgamation theories are such a class (and that, answering a question of Conant in \cite{Co15}, non-modular free amalgamation theories exist). More generally, it is to introduce some properties of theories, essentially generalizing the free amalgamation theories and having no known $\mathrm{NSOP}_{4}$ counterexamples, which imply the $\mathrm{NSOP}_{1} - \mathrm{SOP}_{3}$ dichotomy.

On the other hand, our understanding of \textit{independence} in the $\mathrm{NSOP}$ region beyond $\mathrm{NSOP}_{2}$ theories has remained thin to non-existent. Kaplan and Ramsey (\cite{KR17}) have successfully introduced the concept of Kim-independence, or forking independence ``at a generic scale," as the appropriate extension of forking-independence to the class $\mathrm{NSOP}_{1}$. Yet to extend Kim-independence itself past $\mathrm{NSOP}_{2}$ remains open. Stronger and often stationary abstract independence relations with no known concrete model-theoretic characterization are also abundant in the class $\mathrm{NSOP}_{4}$. The theory of purely abstract independence relations is introduced by Adler in \cite{A09}, where he outlines axioms these relations can satisfy to behave in certain ways like forking-independence in stable theories. In \cite{D19}, D'Elb\'ee proposes the problem of finding a model-theoretic definition of stronger ``free amalgamation" relations alongside Kim-independence in $\mathrm{NSOP}_{1}$ theories, such as the \textit{strong independence} in the theory $\mathrm{ACFG}$ of algebraically closed fields with a generic additive subgroup. He also notes that relations with similar properties to the ``strong independence" found in many $\mathrm{NSOP}_{1}$ theories hold in the strictly $\mathrm{NSOP}_{4}$ Henson graphs. Just as in the case of free amalgamation of generic functional structures in \cite{KR18} or generic incidence structures in \cite{CoK19}, d'Elb\'ee observes that these stronger independence relations can be used to prove the equivalence of forking and dividing for complete types in many known $\mathrm{NSOP}_{1}$ theories. Conant \cite{Co15} introduces his formulation of free amalgamation based on concepts used to study the isometry groups of Urysohn spheres in \cite{ZT12}. There, Conant gives an abstract set of axioms for independence relations generalizing those found in homogeneous structures. These include the relations given by adding no new edges in the (simple) theory of the random graph or the (strictly $\mathrm{NSOP}_{4}$) theory of the generic triangle-free random graph. Aside from the canonical coheirs introduced by the author in \cite{NSOP2} to simulate the assumption of a stationary independence relation in the proof of $\mathrm{NSOP}_{1}$ for $\mathrm{NSOP}_{2}$ theories, our understanding of this phenomenon of ``strong independence" is entirely synthetic. Yet theories exhibiting this phenomenon of ``strong independence" often come equipped with a weaker notion of independence appearing to generalize forking-independence in simple theories or Kim-independence in $\mathrm{NSOP}_{1}$ theories. We show this relation to have a purely model-theoretic characterization as forking-independence ``at a \textit{maximally} generic scale" (in other words, the result of forcing Kim's lemma onto Kim-independence). Our definition extends that of Kim-independence in $\mathrm{NSOP}_{1}$ theories. This new extension of Kim-independence is based on the concept of ``strong Kim-dividing" introduced by Kaplan, Ramsey and Shelah in \cite{KRS19} in the context of ``dual local character" in $\mathrm{NSOP}_{1}$ theories. We show that $\mathrm{NSOP}_{4}$ theories are the last class in the $\mathrm{NSOP}_{n}$ hierarchy where this notion of independence can be symmetric, providing the beginnings of a theory of independence beyond $\mathrm{NSOP}_{1}$.

An outline of the paper is as follows. In section $3$, we introduce a weak set of axioms on stationary independence relations, essentially generalizing the ``freedom" axiom in Conant's free amalgamation theories beyond the traditional homogeneous structures. It is not a true generalization of Conant's axioms, as Conant employs a non-standard definition of stationarity, yet these relations can be found in all of Conant's examples. We show that under these axioms, we can carry out arguments for $\mathrm{NSOP}_{1}$ theories from Chernikov and Ramsey \cite{CR15} and Ramsey and Kaplan \cite{KR17} relative to an independence relation. Significantly, these relative arguments apply even outside of the context of an $\mathrm{NSOP}_{1}$ theory. Relatively to an abstract stationary relations satisfying the weak axioms, we prove the equivalence of two analogues of \cite{KR17}'s characterization of $\mathrm{NSOP}_{1}$. First, a relative Kim's lemma, or representing the least class among invariant Morley sequences in the \textit{dividing order} introduced by \cite{BYC07}; and second, symmetry for relative Kim-independence. It follows that when the relative Kim-independence is symmetric, it is no longer a relative notion, but rather an absolute notion of forking-independence ``at a maximally generic scale." We call this \textit{Conant-independence}. The reason for this name is Conant's observation in \cite{Co15} (Lemma 7.6) that Morley sequences in a free amalgamation relation can only witness dividing when $A\nind^{a}_{C} B$, where $A\ind^{a}_{C} B$ is defined by $\mathrm{acl}(AC) \cap \mathrm{acl}(BC)=\mathrm{acl}(C)$. In other words, Conant shows triviality (and thus symmetry) for the relative Kim-independence in free amalgamation theories; using Conant's argument, we will end up showing that Conant-independence is trivial in free amalgamation theories. A version of Conant-independence was defined by the author in \cite{NSOP2} as a candidate for Kim-independence in $\mathrm{NSOP}_{2}$ theories, but we define it here in terms of invariant Morley sequences rather than coheir Morley sequences, as in the ``strong Kim-dividing" of \cite{KRS19}. We also show that when these equivalent ``relative $\mathrm{NSOP}_{1}$" conditions hold for a relation with our axioms, or more generally when we have symmetry for Conant-independence and enough least elements in the Kim-dividing order even without these axioms, a theory must be either $\mathrm{NSOP}_{1}$ or $\mathrm{NSOP}_{3}$. Additionally, generalizing arguments from \cite{Co15}, such a theory must be either simple or $\mathrm{TP}_{2}$. Importantly, we do not know whether there is an $\mathrm{NSOP}_{4}$ theory where either of these two symmetry or witnessing conditions fail. Our argument uses part of the proof from \cite{NSOP2} of the equivalence of $\mathrm{NSOP}_{1}$ and $\mathrm{NSOP}_{2}$. This in turn adapts to a general setting the arguments from \cite{Co15} on modular free amalgamation theories.

In section 4, we extend Conant's result in \cite{Co15} that modular free amalgamation theories must be either simple or $\mathrm{SOP}_{3}$ to all free amalgamation theories. That is, using the results of the previous section, we show that free amalgamation theories must be either $\mathrm{NSOP}_{1}$ or $\mathrm{SOP}_{3}$. Unlike the case of modular free amalgamation theories, this allows for strictly $\mathrm{NSOP}_{1}$ theories. Conant asks in \cite{Co15} whether non-modular free amalgamation theories exist, and we answer this question positively. We show that Kruckman and Ramsey's example of the generic theory of a function from \cite{KR18}, when equipped with a nonstandard free amalgamation relation that actually falls under Conant's axioms, gives an example of a non-modular free amalgamation theory. As a corollary, we get a converse to Conant's result that a simple free amalgamation theory must be modular, showing that a modular $\mathrm{NSOP}_{1}$ free amalgamation theory must be simple. In a personal communication, Conant noted to the author that Claim 1 of Theorem 7.7 of \cite{Co15} contained a minor error, corrected in \cite{CKru23}. Our proof that a free amalgamation theory must be $\mathrm{NSOP}_{1}$ or $\mathrm{SOP}_{3}$ is based on \cite{NSOP2}, where the analogous claim to Claim 7.7.1 of \cite{Co15} uses either of two arguments which differ entirely from Conant's proof in \cite{Co15} \footnote{The second is due to the participants of the Yonsei University logic seminar and can be found in \cite{Lee22} and footnote 1 of \cite{INDNSOP3}. It uses the proof of Proposition 3.14 of \cite{KR17}}. So we recover Conant's theorem that a modular free amalgamation theory must be simple or $\mathrm{SOP}_{3}$.

In section 5, we give some examples of theories with a ``relatively $\mathrm{NSOP}_{1}$" stationary independence relation with our axioms, and characterize Conant-independence in these theories. We show that the finite-language case of the $\omega$-categorical Hrushovski constructions of \cite{Ev02}, which Conant notes are not necessarily free amalgamation theories in his sense, do satisfy this more general notion of free amalgamation. Using the free-amalgamation relation, we show that Conant-independence gives us a purely model-theoretic interpretation of the \textit{$d$-independence} of \cite{Ev02}, even outside of the simple case (where $d$-independence coincides with forking-independence). We then give a similar analysis to the generic graphs without small cycles, introduced in \cite{She95} as examples of strictly $\mathrm{NSOP}_{4}$ theories. It appears that the curve-excluding fields introduced in upcoming work of Johnson, Walsberg and Ye (\cite{Joh21}) might also have a stationary independence relation with the required properties, with Conant-independence coinciding with algebraic indepenence in the sense of fields. This suggests that these fields must be either strictly $\mathrm{NSOP}_{1}$ or, taking into account the next paragraph, strictly $\mathrm{NSOP}_{4}$.

In section 6, we show that any theory where Conant-independence is symmetric must be $\mathrm{NSOP}_{4}$. The original suggestion of a special significance for $\mathrm{NSOP}_{4}$, in connection with free amalgamation, is due to Patel (\cite{Pat06}), who in unpublished work provided an argument for $\mathrm{NSOP}_{4}$ for various examples. Patel's argument was later generalized, along with work from various other authors, by Conant in \cite{Co15} (where a more complete historical background can be found.) By showing $n = 4$ is the least $n$ so that there is a strictly $\mathrm{NSOP}_{n}$ theory with symmetric Conant-indepednence, we give neostability-theoretic justification for this significance. We then pose some questions about symmetry for Conant-independence within the neostability hierarchy, highlighting some connections with established open problems on dividing lines as well as a potential characterization of $\mathrm{NSOP}_{4}$ in terms of Conant-independence, similar to Kaplan and Ramsey's characterization of $\mathrm{NSOP}_{1}$ in terms of Kim-independence.

\section{Preliminaries}
Notations are standard; $M$ will denote a model while $a, b, c, A, B, C$ will denote sets. A \textit{global type} $p(x)$ is a complete type over the sufficiently saturated model $\mathbb{M}$. For $M \prec \mathbb{M}$, a global type $p(x)$ is \textit{invariant} over $M$ if whether $\varphi(x, b)$ belongs to $p$ for $\varphi(x, y)$ a fixed formula without parameters depends only on the type of the parameter $b$ over $M$ and not on the specific realization of that type. A special subclass of types invariant over $M$ is that of those \textit{finitely satisfiable} over $M$, meaning any formula in the type is satisfied by some element of $M$. We say an infinite sequence $\{b_{i}\}_{i \in I}$, is an \textit{invariant Morley sequence} over $M$ if there is a fixed global type $p(x)$ invariant over $M$ so that $b_{i} \models p(x)|_{M\{b_{j}\}_{j < i}}$ for $i \in I$. This is also said to be an invariant Morley sequence over $M$ \textit{in} the $M$ invariant-type $p(x)$. Invariant Morley sequences over $M$ are indiscernible over $M$, and the  EM-type of an invariant Morley sequence over $M$ depends only on $p(x)$.

We recall Conant's definition of \textit{free amalgamation theories} in \cite{Co15}, and define a few other properties of relations between sets. Many of these definitions come originally from Adler (\cite{A09}) and the axiom system itself resembles that of Ziegler and Tent in \cite{ZT12}. A theory is a \textit{free amalgamation theory} if there is a ternary relation $\ind$ between two sets over another set with the following properties:

Invariance: Whether $A\ind_{C} B$ is an invariant of the type of $ABC$.

Monotonicity: If $A \ind_{C} B$ and $A_{0} \subseteq A$, $B_{0} \subseteq B$, then $A_{0} \ind_{C} B_{0}$.

Full transitivity: For any $A$, if $D \subseteq C \subseteq B$ then $A \ind_{D} B $ if and only if $A \ind_{D} C$ and $A \ind_{C} B$.

Full existence: For any $a, B$ and for $C$ algebraically closed, there is some $a' \equiv_{C} a$ with $a' \ind_{C} B$.

Stationarity: For $a, b, C$ algebraically closed with $C \subseteq a \cap b$, and for any $a' \equiv_{C} a$, if $a \ind_{C} b$ and $a' \ind_{C} b$ then $a' \equiv_{b} a$.

Freedom: For $A, B, C, D$ with $A \ind_{C} B$, if $C \cap AB \subseteq D \subseteq C$, then $A \ind_{D} B$.

Closure: For $a, b, C$ algebraically closed with $C \subseteq a \cap b$ and $a \ind_{C} b$, $ab$ is algebraically closed.

Sometimes a relation is defined only between sets over a model, rather than over an arbitrary set. We define some additional properties that we will use in this case. As Conant's definition of stationarity is nonstandard, this includes the standard formulation of stationarity, which will apply to example 3.2.1 of \cite{Co15}, the random graphs, Henson graphs and Urysohn sphere. 

Full stationarity: If $A\ind_{M} B $, $A'\ind_{M} B $, and $A \equiv_{M} A'$, then $A \equiv_{MB} A'$.

Left extension: If $A\ind_{M} B $ and $A \subseteq C$, there is some $B' \equiv_{A} B$ with  $C\ind_{M} B' $.

Right extension: If $A\ind_{M} B $ and $B \subseteq C$, there is some $A' \equiv_{B} A$ with  $A'\ind_{M} C $.

Finally, we will define one more property of a ternary relation $\ind$ defined over sets:

Finite character: $A \ind_{B} C$ holds whenever $A_{0} \ind_{B} C_{0}$ holds for all finite $A_{0} \subseteq A$, $C_{0} \subseteq C$.

We define $a \ind^{i}_{M} b$ to mean that $\mathrm{tp}(a/Mb)$ extends to an $M$-invariant global type. The relation $a \ind^{a}_{M} b$, denoting $\mathrm{acl}(aM) \cap \mathrm{acl}(bM)=M$ can be found in \cite{Co15}; it is well-known to satisfy right (and left) extension.

We review the relevant regions of the generalized stability hierarchy. The following, which we take as the definition of simplicity, is well-known:

\begin{definition}
We say $\mathrm{tp}(a/bM)$ \emph{does not divide} over $M$, denoted $a\ind^{div}_{M}b$, if there is no formula $\varphi(x, b) \in \mathrm{tp}(a/bM)$ and $M$-indiscernible sequence $\{b_{i}\}_{i \in I}$ starting with $b$ so that $\{\varphi(x, b_{i})\}_{i \in I}$ is inconsistent. A theory $T$ is simple if $\ind^{div}$ is symmetric.
\end{definition}

The properties $\mathrm{NSOP}_{1}$ and $\mathrm{NSOP}_{2}$ were introduced in \cite{DS04}:

\begin{definition}
A theory $T$ is $\mathrm{NSOP}_{1}$ if there does not exist a formula $\varphi(x, y)$ and tuples $\{b_{\eta}\}_{\eta \in 2^{<\omega}}$ so that $\{\varphi(x, b_{\sigma \upharpoonleft n})\}_{n < \omega}$ is consistent for any $\sigma \in 2^{\omega}$, but for any $\eta_{2} \unrhd \eta_{1} \smallfrown \langle 0\rangle$, $\{\varphi(x, b_{\eta_{2}}), \varphi(x, b_{\eta_{1} \smallfrown \langle 1\rangle})\}$ is inconsistent. Otherwise it is $\mathrm{SOP}_{1}$.
\end{definition}

\begin{definition}
A theory $T$ is $\mathrm{NSOP}_{2}$  if there does not exist a formula $\varphi(x, y)$ and tuples $\{b_{\eta}\}_{\eta \in 2^{<\omega}}$ so that $\{\varphi(x, b_{\sigma \upharpoonleft n})\}_{n < \omega}$ is consistent for any $\sigma \in 2^{\omega}$, but for incomparable $\eta_{1}$ and $\eta_{2}$, $\{\varphi(x, b_{\eta_{1}}), \varphi(x, b_{\eta_{2}})\}$ is inconsistent. Otherwise it is $\mathrm{SOP}_{2}$.
\end{definition}

These two classes coincide; see \cite{NSOP2}.

Justifying the ``order" terminology, the following family of classes was introduced in \cite{She95}:

\begin{definition}
Let $n \geq 3$. A theory $T$ is $\mathrm{NSOP}_{n}$ (that is, does not have the \emph{n-strong order property}) if there is no definable relation $R(x_{1}, x_{2})$ with no $n$-cycles, but with tuples $\{a_{i}\}_{i < \omega}$ with $\models R(a_{i}, a_{j})$ for $i <j$. Otherwise it is $\mathrm{SOP}_{n}$.
\end{definition}

We will only concern ourselves with $\mathrm{NSOP}_{n}$ theories for $1 \leq n \leq 4$. Finally, \cite{Sh80} introduces the following notion, whose interaction with the $\mathrm{NSOP}_{n}$ hierarchy beyond $\mathrm{NSOP}_{2}$ remains open:

\begin{definition}
A theory $T$ is $\mathrm{NTP}_{2}$ (that is, does not have the \emph{tree property of the second kind}) if there is no array $\{b_{ij}\}_{i, j < \omega}$, formula $\varphi(x, y)$ and fixed $k \geq 2$ so that there is some fixed $k$ so that, for all $i$, $\{\varphi(x, b_{ij})\}_{j < \omega}$ is $k$-inconsistent, but for any $\sigma < \omega^{\omega}$, $\{\varphi(x, b_{i\sigma(i)})\}_{i < \omega}$ is consistent.
\end{definition}

Kaplan and Ramsey (\cite{KR17}) extend the theory of forking-independence in simple theories to $\mathrm{NSOP}_{1}$ theories. We give a brief overview, mostly by way of motivation:

\begin{definition}
A formula $\varphi(x, b)$ \emph{Kim-divides} over $M$ if there is an invariant Morley sequence $\{b_{i}\}_{i < \omega}$ over $M$ starting with $b$ (said to \emph{witness} the Kim-dividing) so that $\{\varphi(x, b_{i})\}_{i < \omega}$ is inconsistent. A formula  $\varphi(x, b)$ \emph{Kim-forks} over $M$ if it implies a (finite) disjunction of formulas Kim-dividing over $M$. We write $a \ind^{K}_{M} b$, and say that $a$ is \emph{Kim-independent} from $b$ over $M$ if $\mathrm{tp}(a/Mb)$ does not include any formulas Kim-forking over $M$.
\end{definition}

For $\varphi(x, b)$ a formula with parameters, invariant Morley sequence $\{b_i\}_{i < \omega }$ over $M$, $b_0 = b$, is said to \textit{witness} Kim-dividing of $\varphi(x, b)$ (over $M$) if $\{\varphi(x, b_i)\}_{i < \omega}$ is inconsistent. Any $\mathrm{NSOP}_{1}$ theory is characterized by the following variant of Kim's lemma for simple theories, as well as by symmetry of Kim-independence. 

\begin{fact}\label{1-kimslemma}
(\cite{KR17}) Let $T$ be $\mathrm{NSOP}_{1}$. Then for any formula $\varphi(x,b)$ Kim-dividing over $M$, any invariant Morley sequence over $M$ starting with $b$ witnesses Kim-dividing of $b$ over $M$. Conversely, suppose that for any formula $\varphi(x,b)$ Kim-dividing over $M$, any invariant Morley sequence over $M$ starting with $b$ witnesses Kim-dividing of $b$ over $M$. Then $T$ is $\mathrm{NSOP}_{1}$. (The theory $T$ is even $\mathrm{NSOP}_{1}$ if we assume that for any formula $\varphi(x,b)$ Kim-dividing over $M$, any invariant Morley sequence over $M$ starting with $b$ \textit{in an $M$-finitely satisfiable type} witnesses Kim-dividing of $b$ over $M$.

It follows that Kim-forking coincides with Kim-dividing in any $\mathrm{NSOP}_{1}$ theory.
\end{fact}

\begin{fact}\label{1-symm}
(\cite{CR15}, \cite{KR17}) The theory $T$ is $\mathrm{NSOP}_{1}$ if and only if $\ind^{K}$ is symmetric.
\end{fact}

The following preorder restricts the \textit{dividing order} of \cite{BYC07}. We are interested in the least class, when it exists. 

\begin{definition}
Fix a type $r(x)$ over a model $M$. We define the following order on global $M$-invariant extensions of $r(x)$. Let $p(x), q(x)$ be two $M$-invariant global types extending $r(x)$. Then \emph{$p(x)$ is greater than or equal to $q(x)$ in the Kim-dividing order} if for all formulas $\varphi(x, b)$ where $b \models r(x) = p(x)|_M = q(x)|_M$, if Morley sequences in $q(x)$ witness Kim-dividing of $\varphi(x, b)$ over $M$, then Morley sequences in $p(x)$ witness Kim-dividing of $\varphi(x, b)$ over $M$.

\end{definition}

We will say an $M$-invariant type $p(x)$ is \emph{least} in the Kim-dividing order if, in the Kim-dividing order, it is less than or equal to all $M$-invariant types whose restriction to $M$ is $p(x)|_M$.

\section{Generalized free amalgamation and relative $\mathrm{NSOP}_{1}$}
Throughout this section we assume unless otherwise noted a ternary relation $\ind$ between sets is defined over models and has invariance, monotonicity, full existence over models, and full stationarity. Following Definition 7.5 of \cite{Co15}, we first define special Morley sequences.

\begin{definition}\label{1-morleysequence}
Let $M \prec \mathbb{M}$. An $\ind$\emph{-Morley sequence} over $M$ is an infinite sequence $\{a_{i}\}_{i \in I}$ (for $I$ an infinite set) so that $a_{i} \ind_{M} a_{< i}$ for all $i \in I$ and $a_{i} \equiv_{M} a_{j}$ for $i, j < \omega$. If $p= \mathrm{tp}(a_i/M)$ then it is an $\ind$\emph{-Morley sequence in} $p$ over $M$.
\end{definition}

We list some basic facts on $\ind$-Morley sequences, where $\ind$ satisfies our assumptions; they follow easily from these assumptions.

\begin{fact}\label{1-morleysequencebasics}

Let $\ind$ satisfy the assumptions at the beginning of this section.

(1) For every type $p$ over a model $M$, there is a unique global extension $p^{*}$ so that $a \ind_{M} B$ for any set $B$ and $a \models p^{*}|_{MB}$. The type $p^{*}$ is invariant over $M$.

(2) Any $\ind$-Morley sequence $\{b_{i}\}_{i \in I}$ in a type $p$ over $M$ is an invariant Morley sequence in $p^{*}$ over $M$. So it is $M$-indiscernible, and its EM-type is determined by $p$.

(3) $a \ind_{M} b \Rightarrow a \ind^{a}_{M} b$ 
    
\end{fact}

\begin{proof}
    These are easy exercises. We sketch (3), which also follows from the well-known fact that $\ind^{i}$ implies $\ind^{a}$.
    
    (3) It is well known that there is $b' \equiv_{M} b$ with $b' \ind^{a}_{M} b$; by our assumptions on $\ind$ and an automorphism, we can find such $b'$ so that $a \ind_{M} bb'$. Also by our assumptions and an automorphism, we then have $\mathrm{acl}(a) \ind_{M} bb'$ (see e.g. \cite{A09}.) Then $\mathrm{acl}(a) \cap \mathrm{acl}(b) = \mathrm{acl}(a) \cap \mathrm{acl}(b') \subseteq \mathrm{acl}(b) \cap \mathrm{acl}(b') \subseteq M$.
\end{proof}

We consider a new axiom on $\ind$, motivated by the freedom axiom from \cite{Co15} defined in section 2 and covering all of the examples from \cite{Co15}.

\begin{definition}
Let $\ind$ satisfy the assumptions at the beginning of this section. Then $\ind$ satisfies the \emph{generalized freedom axiom} if the following holds:

If $M \prec M' \prec \mathbb{M}$ and there is an $\ind$-Morley sequence over $M$ starting with $a$ and indiscernible over $M'$, then an $\ind$-Morley sequence starting with $a$ over $M'$ is also an $\ind$-Morley sequence over $M$.    
\end{definition}

(See \cite{KR19}, \cite{DKR22} for some results involving preservation of Morley sequences under change of base.)

\begin{remark}\label{1-freedomgeneralizedfreedom}
If $\ind$ additionally satisfies the freedom axiom over models, it also satisfies the generalized freedom axiom.
\end{remark}

\begin{proof}
Since any two terms of $\ind$-Morley sequences over $M$ starting with $a$ will be $\ind^{a}$-independent over $M$, the hypothesis of the generalized freedom axiom implies $M' \ind^{a}_{M} a$. The rest is just the proof of Lemma 7.6 of \cite{Co15}. By stationarity, it suffices to construct, for any infinite index set $I$, an $\ind$-Morley sequence $\{a_{i}\}_{i \in I}$ over $M'$ starting with $a$ that remains an $\ind$-Morley sequence over $M$. By Fact \ref{1-morleysequencebasics}.1, the property of being an $\ind$-Morley sequence in a fixed type over a model is type-definable, so by compactness, we can assume $I = \omega$. Suppose $a_{0}, \ldots, a_{n}$ are already constructed. Using full existence, find $a_{n+1} \equiv_{M} a$ with $a_{n+1} \ind_{M'} a_{0} \ldots a_{n} $. So $M' \cap a_{0} \ldots a_{n+1} \subseteq M \subseteq M'$. Then by the freedom axiom, additionally $a_{n+1} \ind_{M} a_{0} \ldots a_{n} $.

\end{proof}

\begin{example}\label{1-freeamalgamationtheories}
In Examples 3.2.1(i-iii) of \cite{Co15}, the random graphs, Henson graphs and model companion of the $\{0, 1, 2\}$-valued metric spaces, free amalgamation satisfies full stationarity and therefore satisfies the generalized freedom axiom.
\end{example}

\begin{example}\label{1-notfullystationary}
In the generic $(K_{n}+K_{3})$-free graphs of \cite{CSS99} (the first of which is introduced in \cite{Ko99}), it follows from the discussion in Example 3.2.2 of \cite{Co15} (namely the result of Patel \cite{Pat06} that the class of $(K_{n}+K_{3})$-free graphs is closed under free amalgamation over an algebraically closed base; since the algebraic closure is distintegrated, this free amalgamation is itself algebarically closed) that isomorphic \textit{algebraically closed} sets are elementarily equivalent. Since it is required for elementary equivalence that the sets be algebraically closed, the free amalgamation from this example only satisfies stationarity, rather than full stationarity. However, consider the fully stationary relation $A\ind_{M}B$ defined by free amalgamation of $\mathrm{acl}(AM) $ and  $\mathrm{acl}(BM) $ over $M$; we show that the generalized freedom axiom holds. Suppose the hypothesis holds, so $M' \ind^{a}_{M} a$. Consider an $\ind$-Morley sequence $\{a_{i}\}_{i < \omega}$ over $M'$ starting with $a$. Then $\{\mathrm{acl}(Ma_{i})\}_{i < \omega}$ can be seen to be $\ind^{a}$-independent over $M$, and because $\mathrm{acl}(Ma_{i})$ does not meet $M'$ except in $M$, that $\{\mathrm{acl}(M'a_{i})\}_{i < \omega}$ are in free amalgamation (given by adding no new edges) over $M'$ implies that the $\{\mathrm{acl}(Ma_{i})\}_{i < \omega}$ are in free amalgamation over $M$.
\end{example}

We consider Conant's other example from \cite{Co15}, the freely disintegrated $\omega$-categorical Hrushovski constructions, in Section 5, as part of the larger general class of $\omega$-categorical Hrushovski constructions defined in \cite{Ev02}. Conant notes that, while the freely-disintegrated $\omega$-categorical Hrushovski constructions in \cite{Ev02} are free amalgamation theories, the $\omega$-categorical Hrushovski constructions defined in \cite{Ev02} are not, in general, free amalgamation theories.

\begin{example}\label{1-strongindependence}
In the strictly $\mathrm{NSOP}_{1}$ theory $\mathrm{ACFG}$ of algebraically closed fields with a generic additive subgroup, the strong independence relation $A\ind^{\mathrm{st}}_{M}B$, introduced as part of a larger family in \cite{D18} and developed in \cite{D19}, given by $A \ind^{\mathrm{ACF}}_{M} B$ and $G(\mathrm{acl}(MAB))=G(\mathrm{acl}(MA))+G(\mathrm{acl}(MB))$, satisfies the generalized freedom axiom. Note that Kim-independence $A\ind^{K}_{M}B$ is given by the ``weak independence" $A \ind^{\mathrm{ACF}}_{M} B$ and $G(\mathrm{acl}(MA)+\mathrm{acl}(MB))=G(\mathrm{acl}(MA))+G(\mathrm{acl}(MB))$, and the hypothesis of this axiom in the $\mathrm{NSOP}_{1}$ case is just Kim-independence. It is expected that all of the other examples from the literature of ``strong independence" in $\mathrm{NSOP}_{1}$ theories listed in \cite{D18} also satisfy this axiom.
\end{example}

We wish to show that even outside of the $\mathrm{NSOP}_{1}$ context, the theory of Kim-forking from \cite{KR17} characteristic of $\mathrm{NSOP}_{1}$ theories can be developed relative to an independence relation $\ind$ satisfying the generalized freedom axiom, though when the equivalent relative versions of $\mathrm{NSOP}_{1}$ are satisfied, the relative version of Kim-independence becomes a new absolute independence relation. We first introduce the relative notion:

\begin{definition}
Let $\varphi(x, b)$ be a formula. We say $\varphi(x, b)$ \emph{$\ind$-Kim-divides} over a model $M$ if $\{\varphi(x, b_{i})\}_{i\in I}$ is inconsistent (so $k$-inconsistent for some $k$) for some (any) $\ind$-Morley sequence $\{b_{i}\}_{i \in I}$ over $M$ starting with $b$, and that it \emph{$\ind$-Kim-forks} over $M$ if it implies a (finite) disjunction of formulas $\ind$-Kim-dividing over $M$. We say $a$ is $\ind$-Kim-independent from $b$ over $M$ (written $a\ind^{K^{\ind}}_{M}b$) if $a$ does not satisfy a formula of the form $\varphi(x, b)$ $\ind$-Kim-forking over $M$.
\end{definition}

Analogously to \cite{TZ12}, we have the following fact:

\begin{fact}\label{1-tentziegler}

Let $a \ind^{K^{\ind}}_{M} b$. Then there is some $\ind$-Morley sequence $I$ over $M$ starting with $b$ that is indiscernible over $a$.

Moreover, if there is some $\ind$-Morley sequence $I$ over $M$ starting with $b$ that is indiscernible over $a$, then $\mathrm{tp}(a/Mb)$ contains no formulas $\ind$-Kim-dividing over $M$.

\end{fact}

\begin{proof}
    This is standard. Choose any $\ind$-Morley sequence $I'=\{b'_i\}_{i<\omega}$ over $M$ starting with $b$. Then by compactness there is some $a'$ so that $a'b'_i \equiv_M ab$ for $i < \omega$. By an automorphism, we can find $I'' =\{b''_i\}_{i < \omega}$ so that $ab''_i \equiv_M ab$ for $i < \omega$, and by Ramsey's theorem and compactness, we can assume $I''$ to be indiscernible over $Ma$. Then by an automorphism, we can additionally get $b''_0 = b$, yielding the desired $I$.

For the second clause, clearly, for every formula $\varphi(x, b) \in \mathrm{tp}(a/Mb)$, $I$ as in the hypothesis does not witness $\ind$-Kim-dividing over $M$. By Fact \ref{1-morleysequencebasics}.2, no $\ind$-Morley sequence over $M$ in $\mathrm{tp}(b/M)$ witnesses Kim-dividing of $\varphi(x, b)$ over $M$.
\end{proof}

A feature of stationarity is that we automatically get ``Kim's lemma" (the analogue of Fact 2.1) for the class of $\ind$-Morley sequences taken \textit{alone}, giving us equivalence of $\ind$-Kim-forking and $\ind$-Kim-dividing with no further assumptions.

\begin{proposition}\label{1-relativeforkingdividing}
For formulas, $\ind$-Kim-forking coincides with $\ind$-Kim-dividing. Moreover, $\ind^{K^{\ind}}$ satisfies right extension.
\end{proposition}

\begin{proof}
The following is standard; see \cite{KR17} for the application of this method to Kim-independence in $\mathrm{NSOP}_{1}$ theories. Let $\models \varphi(x, b)\rightarrow \bigvee_{i = 1}^{n} \psi_{i}(x, c_{i})$ for $\psi_{i}(x, c_{i})$ $\ind$-Kim-dividing over $M$. By left extension (which follows from the assumptions) and monotonicity for $\ind$, whether or not a formula $\ind$-Kim-divides over $M$ does not change when adding unused parameters, so we can assume $c_{i} = b$ for $1 \leq i \leq n$. Then $\varphi(x, b)$ Kim-divides over $M$, for suppose otherwise. Let $\{b_{i}\}_{i < \omega}$ be an $\ind$-Morley sequence over $M$ starting with $b$; then there will be some $a$ realizing $\{\varphi(x, b_{i})\}_{i < \omega}$. So by the pigeonhole principle, there will be some $1 \leq j \leq n$ so that $a$ realizes $\{\psi_{j}(x, b_{i})\}_{i \in S}$ for $S \subseteq \omega$ infinite. But by monotonicity and an automorphism, we can assume $\{b_{i}\}_{i \in S}$ is an $\ind$-Morley sequence over $M$ starting with $b$, contradicting $\ind$-Kim-dividing of $\psi_{j}(x, b)$.

Right extension is also standard: if $a \ind^{K^{\ind}}_{M} b$ and there is no $a' \equiv_{Mb} a$ with $a' \ind^{K^{\ind}}_{M} bc$, then by compactness, some formula $\varphi(x, b) \in \mathrm{tp}(a/Mb)$ implies a disjunction of $Mbc$-formulas $\ind$-Kim-forking over $M$, so $\varphi(x, b)$ itself $\ind$-Kim-forks over $M$, contradicting $a \ind^{K^{\ind}}_{M} b$.
\end{proof}

Next, we introduce one possible formulation of $\mathrm{NSOP}_{1}$ relative to $\ind$ (see Fact \ref{1-kimslemma}).

\begin{definition}\label{1-relativekimslemma}
The relation $\ind$ satisfies the \emph{relative Kim's lemma} if for any model $M$ and type $p$ over $M$, and for $p^{*}$ the unique global $M$-invariant type so that $a \ind_{M} B$ for $a \models p^{*}|_{MB}$ (Fact \ref{1-morleysequencebasics}), $p^{*}$ is least in the Kim-dividing order among $M$-invariant extensions of $M$.
\end{definition}

Aside from the motivation by $\mathrm{NSOP}_{1}$ theories (as well as similarity to a property of the canonical coheirs of \cite{NSOP2}), this is a natural assumption. Strictly $\mathrm{NSOP}_{4}$ theories are often defined as the generic examples of structures avoiding a particular configuration, such as the Henson graph avoiding $K_{n}$ or the $\omega$-categorical Hrushovski constructions avoiding finite substructures of negative predimension. Free amalgamation-like relations in these examples will have the least amount of obstructions to consistency along an invariant Morley sequence, which is to say, obstructions (say, edges or relations) to the avoidance of a forbidden configuration. Using the generalized freedom axiom, arguments from \cite{CR15}, \cite{KR17} can be carried out here, showing the equivalence of this assumption to symmetry of the relative Kim-independence (see Facts \ref{1-kimslemma} and \ref{1-symm}).

\begin{theorem}\label{1-relativensop1}
Suppose $\ind$ satisfies the generalized freedom axiom. Then $\ind$ satisfies the relative Kim's lemma if and only if $\ind^{K^{\ind}}$ is symmetric.
\end{theorem}

\begin{proof}

We follow the proofs of Theorems 3.16 and 5.16 of \cite{KR17}, taking note of where the generalized freedom axiom applies in each direction; note that because the $\ind$-Morley sequences will go in the opposite direction of the configurations originally found in the proofs of the results on $\mathrm{NSOP}_{1}$ theories, we will require densely ordered indiscernible sequences. We will also need to make some modifications to respect the Skolemization.

($\Rightarrow$) Suppose $\ind$ also satisfies the relative Kim's lemma. Then we have the following chain condition: 
\begin{claim}\label{1-cc}
(Chain Condition) Let $a \ind^{K^{\ind}}_{M} b$. Then there is some $\ind$-Morley sequence $I = \{b_{i}\}_{i < \omega}$ over $M$ indiscernible over $Ma$ starting with $b$ so that $a \ind^{K^{*}}_{M} I$.
\end{claim}

\begin{proof}
The proof follows Proposition 5.3 of \cite{NSOP2}, itself similar to the standard proof of the chain condition found in, say, \cite{KR17}. Let $I=\{b_{i}\}_{i < \omega }$ be as in Fact \ref{1-tentziegler}; we show that $a \ind^{K^{*}}_{M} I$. Let $\varphi(x, I) \in \mathrm{tp}(a/MI)$; this can be assumed to have parameters in $b_{0}\ldots b_{k}$ for $k < \omega$. Because $I$ is an invariant Morley sequence over $M$, $\{b_{ik}b_{ik+1} \ldots b_{ik+(k-1)} \}_{i \in \omega}$ is also an invariant Morley sequence over $M$, say in the $M$-invariant type $q$ extending $p = \mathrm{tp}(b_{0}\ldots b_{k}/M) $. Invariant Morley sequences in $q$ over $M$, one of which is $\{b_{ik}b_{ik+1} \ldots b_{ik+(k-1)} \}_{i \in \omega}$ itself, do not witness dividing of $\varphi(x, I)= \varphi(x, b_{0} \ldots b_{k})$ over $M$. Therefore, neither do invariant Morley sequences over $M$ in $p^{*}$, where $p^{*}$ is as in Definition \ref{1-relativekimslemma}, since $p^{*}$ is less than or equal to $q$ in the Kim-dividing order. But invariant Morley sequences in $M$ over $p^{*}$ are just $\ind$-Morley sequences in $\mathrm{tp}(b_{0}, \ldots, b_{k}/M)$, so $\varphi(x, I)$ does not $\ind$-Kim divide over $M$, and by Proposition \ref{1-relativeforkingdividing}, does not $\ind$-Kim-fork over $M$. Since $\varphi(x, I) \in \mathrm{tp}(a/MI)$ was arbitrary, $a \ind^{K^{*}}_{M} I$.
\end{proof}

Now suppose for contradiction that $a \ind^{K^{\ind}}_{M} b$ but $b$ is $\ind$-Kim-dependent on $a$ over $M$. By Proposition \ref{1-relativeforkingdividing}, let $\varphi(x, a) \in \mathrm{tp}(b/Ma)$ $\ind$-Kim-divide over $M$, and choose a Skolemization of $T$.

\begin{claim}\label{1-consistencyinconsistency}
There is a sequence $\{c_{i, 0}, c_{i, 1}\}_{i < \omega}$ with $c_{i,j} \equiv_{M} a$, with $c_{i, 0} \equiv_{\mathrm{dcl}_{\mathrm{Sk}}(Mc_{<i, 0}, c_{<i, 1})} c_{i, 1}$, and a formula $\varphi(x, y)$ with $\{\varphi(x, c_{i, 0})\}_{i < \omega}$ consistent, but $\{c_{i, 1}\}_{i < \omega}$, read backwards, an $\ind$-Morley sequence over $M$; therefore, $\{\varphi(x, c_{i, 1})\}_{i < \omega}$ will be inconsistent. 
\end{claim}

\begin{proof}
This configuration is obtained in a similar way to that from the proof of Proposition 5.13 of \cite{KR17}, but with two differences. First, as in \cite{NSOP2}, the branches must form \textit{special} Morley sequences--in this case $\ind$-Morley sequences instead of canonical Morley sequences--rather than any invariant Morley sequence. Second, the sequence $\{c_{i, 0}, c_{i, 1}\}_{i < \omega}$ is extracted in the Skolemization. Because the construction is similar to \cite{KR17}, we only sketch the proof. The idea is to first construct, for all $n < \omega$, tree-indexed sets $\{a_{\eta}\}_{\eta \in \omega^{< n }}$, with additional leaf nodes $\{a_{\eta}\}_{\eta \in \omega^{n}}$, at the top of each branch, satisfying the following properties:

(1) Let $\eta \in \omega^{< n}$, $\eta' \in \omega^{n}$, $\eta \lhd \eta'$ be a non-leaf node, and a leaf on its branch. Then $a_{\eta}a_{\eta'} \equiv_{M} ab$.

(2) The maximal subtrees at a given node $\eta \in \omega^{< n}$ form an $\ind$-Morley sequence over $M$: for $k < \omega$, $a_{\unrhd \eta \smallfrown \langle k \rangle} \equiv_{M} a_{\unrhd \eta \smallfrown \langle 0 \rangle}$ and $a_{\unrhd \eta \smallfrown \langle k \rangle} \ind_{M} a_{\unrhd \eta \smallfrown \langle n-1 \rangle} \ldots a_{\unrhd \eta \smallfrown \langle 0 \rangle}$.

(3) Each node $\eta$ is $\ind$-Kim independent from all greater nodes: $a_{\eta} \ind_{M}^{K^{\ind}} a_{\rhd \eta }$

This is by induction. We start with $\{b\}$ as our initial stage of the tree. Suppose at a given stage, the tree $I= \{a_{\eta}\}_{\eta \in \omega^{\leq n }}$ is already built. We 
will find the following:

(a) some $a'_{\emptyset} \ind^{K^{\ind}}_{M} I$ so that the type of $a'_{\emptyset}$ with each leaf node over $M$ is the same as that of $ab$ over $M$, and

(b) some $\ind$-Morley sequence $\{I^i\}_{i < \omega}$ with $I^{0} = I$ and $a'_{\emptyset} \ind^{K^{*}} \{I^i\}_{i < \omega}$

Then by reindexing, we get the next stage, $ \{a_{\eta}\}_{\eta \in \omega^{\leq  n+1 }}$.

For the base case where $I = \{b\}$ we get (a) by the assumption that $a \ind_{M}^{K^{\ind}} b$, choosing $a = a_{\emptyset}$. For the inductive step, note that the root node of $I$ is $\ind$-Kim-independent from the rest of the tree by (3) and satisfies, with each leaf node, the type of $ab$ over $M$ by (1). So we get $a'_{\emptyset}$ by right extension for $\ind^{K^{\ind}}$, giving us (a). Next, to get (b), use the chain condition to choose an $\ind$-Morley sequence $\{I^{i}\}_{i < \omega}$ starting with $I$ and indiscernible over $Ma'_{\emptyset}$ so that $a'_{\emptyset} \ind^{K^{\ind}}_{M}\{I^{i}\}_{i < \kappa}$. Then $\{I^{i}\}_{i < \omega}$ is as desired.

Notice that the trees $\{a_{\eta}\}_{\eta \in \omega^{< n }}$ satisfy the following properties:

(1$'$) For $\sigma \in \omega^{\leq n}$, $\{\varphi(x, a_{\sigma|_{n}})\}_{i < n}$ is consistent (because $a_{\sigma} \models \{\varphi(x, a_{\sigma|_{n}})\}_{i < n}$, by (1))

(2$'$) Let $\eta_{1} <_{\mathrm{lex}} \ldots <_{\mathrm{lex}} \eta_{k} \in \omega^{< n}$ with $\eta_{i+1} \wedge \ldots \wedge \eta_{1} \lhd \eta_{i} \wedge \ldots \wedge \eta_{1}$ for $1 \leq i < k$ (i.e. a \textit{descending comb} in the sense of Proposition 2.51, item IIIb, \cite{sim12}). Then $\{a_{\eta_{k}}\}_{i \leq k}$ begin an $\ind$-Morley sequence over $M$ in $\mathrm{tp}(a/M)$ (by (2)); that is, $a_{\eta_{i}} \equiv_{M} a$ and $a_{\eta_i} \ind_{M} a_{\eta_{i -1}} \ldots a_{\eta_{1}}$ for $1 \leq i < k$.

Now both of these conditions are type-definable, the second by Fact \ref{1-morleysequencebasics}.2. So by compactness, we can find a tree $\{c_{\eta}\}_{\eta \in \omega^<\kappa}$ for large $\kappa$ satisfying both of these conditions: (1$'$) for $\sigma \in \omega^{\leq \kappa}$, $\{\varphi(x, c_{\sigma|_{\lambda}})\}_{\lambda < \kappa}$ is consistent, and (2$'$) for $\eta_{1} <_{\mathrm{lex}} \ldots <_{\mathrm{lex}} \eta_{k} \in \omega^{< \kappa}$ with $\eta_{i+1} \wedge \ldots \wedge \eta_{1} \lhd \eta_{i} \wedge \ldots \wedge \eta_{1}$ for $1 \leq i < k$, $\{c_{\eta_{k}}\}_{i \leq k}$ begins an $\ind$-Morley sequence over $M$ in $\mathrm{tp}(a/M)$.

Now choose a Skolemization of $T$; we use the argument of Proposition 5.6 of \cite{CR15}. (This differs slightly from the approach used in sections 5.1-5.2 of \cite{KR17} involving indiscernible trees, which we could also have used.) Suppose the $c_{i, 0}=c_{\lambda_{i}}$, indexed by nodes $\lambda_{i}$ and $c_{i, 1}=c_{\eta_{i}}$ indexed by nodes $\eta_{i}$ with $\eta_{j} \wedge \lambda_{j} \rhd \lambda_{i}$ for $1 \leq i < j \leq n$ and $\lambda_{i} \unrhd  (\eta_{i} \wedge \lambda_{i})\smallfrown \langle 0 \rangle, \eta_{i} \unrhd  (\eta_{i} \wedge \lambda_{i})\smallfrown \langle 1 \rangle $, are already constructed. Then using the pigeonhole principle, choose nodes $\lambda_{n+1} =\lambda_{n}\smallfrown \langle 0 \rangle^{\kappa_{1}} \smallfrown \langle 1\rangle$, $\eta_{n+1} = \lambda_{n}\smallfrown \langle 0 \rangle^{\kappa_{2}} \smallfrown \langle 1\rangle$ for $\kappa_{1} < \kappa_{2} < \kappa$ so that the corresponding terms of the tree, which we then call $c_{n+1, 0}=c_{\lambda_{n+1}}$ and $c_{n+1, 1}=c_{\eta_{n+1}}$, have the same type over $\mathrm{dcl}_{\mathrm{Sk}}(Mc_{\leq n, 0}, c_{\leq n, 1})$.
\end{proof}

We now apply the generalized freedom axiom to carry out the argument for Proposition 3.14 of \cite{KR17}, the one underlying Kim's lemma in actual $\mathrm{NSOP}_{1}$ theories, to contradict the relative Kim's lemma.

We can find $\{c_{i, 0}, c_{i, 1}\}_{i \in \mathbb{Q}^{+}}$ for $\mathbb{Q^{+}}=\mathbb{Q} \cup \{\infty\}$, indiscernible over $M$ in the Skolemization with the same properties. Let $M' = \mathrm{dcl}_{\mathrm{Sk}}(M\{c_{i, 0}, c_{i, 1}\}_{i \in \mathbb{Q}})$, and $p(y)=\mathrm{tp}(c_{\infty, 0}/M')=\mathrm{tp}(c_{\infty, 1}/M')$.

\begin{claim}\label{1-indmorleysequenceexistence}
There is an $\ind$-Morley sequence over $M$ of realizations of $p(y)$.
\end{claim}

\begin{proof}
By compactness, it suffices to show the same replacing $p(y)$ with its restriction to $M_{j}=\mathrm{dcl}_{\mathrm{Sk}}(M\{c_{i, 0}, c_{i, 1}\}_{i < j})$ for some $j \in \mathbb{Q}$. But this is just $\mathrm{tp}(c_{j+1, 1}/M_{j})$, and $\{c_{k, 1}\}_{j < k \leq j+1}^{\infty}$, read backwards, is an $\ind$-Morley sequence starting with $c_{j+1,1}$ indiscernible over $M_{j}$.
\end{proof}

Now just as in the proof of Proposition 3.14 of \cite{KR17}, the consistency of $\{\varphi(x, c_{i, 0})\}_{i \in \mathbb{Q}^{+}}$ gives an $M'$-finitely satisfiable global extension $p'$ of $p$ whose Morley sequences do not witness the Kim-dividing of $\varphi(x, c_{\infty, 0}) $ over $M'$. This type $p'$ is the the limit type of $\{c_{i, 0}\}_{i \in \mathbb{Q}}$ over $M$. But by the generalized freedom axiom and Claim \ref{1-indmorleysequenceexistence}, a $\ind$-Morley sequence starting with $c_{\infty, 1}$ over $M'$ will remain an $\ind$-Morley sequence over $M$. Such a sequence will witness the Kim-dividing of $\varphi(x, c_{\infty, 1}) $ over $M$, and thus over $M'$, by the inconsistency of $\{\varphi(x, c_{i, 1})\}_{i \in \mathbb{Q}}$. Since $c_{\infty, 1} \equiv_{M} c_{\infty, 0}$, this contradicts the relative Kim's lemma.

($\Leftarrow$) Suppose the relative Kim's lemma fails. We will find $a, b, M'$ with $\mathrm{tp}(a/M'b)$ finitely satisfiable and thus invariant over $M'$, so a fortiori $a \ind^{K^{\ind}}_{M'} b$, but with $b \nind^{K^{\ind}}_{M'} a$. This will prove that $\ind^{K^{\ind}}$ is not symmetric, as desired. The failure of the relative Kim's lemma gives us a formula $\varphi(x, c)$ that $\ind$-Kim divides over $M$, and an $M$-invariant extension $q(y)$ of $\mathrm{tp}(c/M)$ whose invariant Morley sequences do not witness this Kim-dividing. Choose a Skolemization of $T$.

\begin{claim}\label{1-consistencyinconsistency2}
We get the same configuration as in Claim \ref{1-consistencyinconsistency}: There is a sequence $\{c_{i, 0}, c_{i, 1}\}_{i \in \mathbb{Z}}$ with $c_{i, 0} \equiv_{\mathrm{dcl}_{\mathrm{Sk}}(Mc_{<i, 0}, c_{<i, 1})} c_{i, 1}$, and a formula $\varphi(x, y)$ with $\{\varphi(x, c_{i, 0})\}_{i \in \mathbb{Z}}$ consistent, but $\{c_{i, 1}\}_{i \in \mathbb{Z}}$, read backwards, an $\ind$-Morley sequence over $M$ in $\mathrm{tp}(c/M)$. Therefore $\{\varphi(x, c_{i, 1})\}_{i \in \mathbb{Z}}$ is inconsistent. 
\end{claim}

\begin{proof}
Attempting the method of Proposition 3.15 of \cite{KR17}, we fail to respect the Skolemization, so we will instead construct a very large tree, say, by induction and compactness. See Lemma 4.5 of \cite{NSOP2}; the construction here will be similar but easier. We will construct trees $\{a_{\eta}\}_{\eta \in \omega^{< n}}$ satisfying the following properties:

(1) For a branch $\sigma \in \omega^{n}$, the branch read backwards, i.e. $a_{\sigma|_n}, a_{\sigma|_{n-1}}, \ldots, a_\emptyset$, begins an invariant Morley sequence in $q(x)$ over $M$, i.e. $a_{\sigma|_{i}} \models q(x)|_{a_{\sigma_{i+1}} \ldots a_{n}}$ for $ 0 \leq i < n$

(2) The maximal subtrees at a given node $\eta \in \omega^{< n}$ form an $\ind$-Morley sequence over $M$: for $k < \omega$, $a_{\unrhd \eta \smallfrown \langle k \rangle} \equiv_{M} a_{\unrhd \eta \smallfrown \langle 0 \rangle}$ and $a_{\unrhd \eta \smallfrown \langle k \rangle} \ind_{M} a_{\unrhd \eta \smallfrown \langle n-1 \rangle} \ldots a_{\unrhd \eta \smallfrown \langle 0 \rangle}$.

Suppose the $n$th stage $I = \{a_{\eta}\}_{\eta \in \omega^{< n}}$ is constructed. We can easily find the following:

(a) An $\ind$-Morley sequence $\{I^{i}\}_{i < \omega}$ with $I^{0}=I$.

(b) A root node $a'_{\emptyset} \models q(x)|_{M_{\{I^{i}\}_{i < \omega}}} $.

Reindexing, we get $I = \{a_{\eta}\}_{\eta \in \omega^{< n+1}}$ satisfying (1) and (2).

Notice that (1) implies (1$'$) and (2) implies (2$'$):

(1$'$) For $\sigma \in \omega^{\leq n}$, $\{\varphi(x, a_{\sigma|_{n}})\}_{i < n}$ is consistent (because Morley sequences in $q(x)$ do not witness Kim-dividing of $\varphi(x, b)$.)

(2$'$) Let $\eta_{1} <_{\mathrm{lex}} \ldots <_{\mathrm{lex}} \eta_{k} \in \omega^{< n}$ with $\eta_{i+1} \wedge \ldots \wedge \eta_{1} \lhd \eta_{i} \wedge \ldots \wedge \eta_{1}$ for $1 \leq i < k$. Then $\{a_{\eta_{k}}\}_{i \leq k}$ begin an $\ind$-Morley sequence over $M$ in $\mathrm{tp}(a/M)$ (by (2)); that is, $a_{\eta_{i}} \equiv_{M} a$ and $a_{\eta_i} \ind_{M} a_{\eta_{i -1}} \ldots a_{\eta_{1}}$ for $1 \leq i < k$.

Just as in the penultimate paragraph of the proof of Claim 3.10, we can then use compactness to get a tree a tree $\{c_{\eta}\}_{\eta \in \omega^<\kappa}$ for large $\kappa$ satisfying both of these conditions: (1$'$) for $\sigma \in \omega^{\leq \kappa}$, $\{\varphi(x, a_{\sigma|_{\lambda}})\}_{\lambda < \kappa}$ is consistent, and (2$'$) for $\eta_{1} <_{\mathrm{lex}} \ldots <_{\mathrm{lex}} \eta_{k} \in \omega^{< \kappa}$ with $\eta_{i+1} \wedge \ldots \wedge \eta_{1} \lhd \eta_{i} \wedge \ldots \wedge \eta_{1}$ for $1 \leq i < k$, $\{a_{\eta_{k}}\}_{i \leq k}$ begins an $\ind$-Morley sequence over $M$ in $\mathrm{tp}(a/M)$.

The claim follows as in the last paragraph of the proof of Claim \ref{1-consistencyinconsistency}.

\end{proof}

Now we follow Proposition 5.6 in Chernikov and Ramsey in \cite{CR15}, the result underlying the other direction of Kaplan and Ramsey's symmetry characterization of $\mathrm{NSOP}_{1}$ in \cite{KR17}. After choosing $\{c_{i, 0}, c_{i, 1}\}_{i \in \mathbb{Z}}$ as in Claim \ref{1-consistencyinconsistency2} to be indiscernible in the Skolemization, let us replace the index set $\mathbb{Z}$ with $\mathbb{Q}^{+}= \mathbb{Q} \cup \{\infty\}$. Let $M'=\mathrm{dcl}_{\mathrm{Sk}}(M\{c_{i, 0}, c_{i, 1}\}_{i \in \mathbb{\mathbb{Q}}})$. By consistency of $\{\varphi(x, c_{i, 0})\}_{i \in \mathbb{\mathbb{Q}^{+}}}$, we find $b \models \{\varphi(x, c_{i, 0})\}_{i \in \mathbb{\mathbb{Q}^{+}}}$. By Ramsey, compactness and an automorphism, we can additionally choose $b$ so that $\{ c_{i, 0}\}_{i \in \mathbb{\mathbb{Q}^{+}}}$ is $Mb$-indiscernible. Let $a=c_{\infty, 0}$ so that $b \models \varphi(x, a)$. Then $\mathrm{tp}(a/M'b)$ is finitely satisfiable over $M'$. It remains to show that $\varphi(x, a)$ $\ind$-Kim-divides over $M'$. By invariance and $a \equiv_{M'} c_{1, \infty}$, it is enough to show that  $\varphi(x, c_{1, \infty})$ $\ind$-Kim-divides over $M'$. The sequence $\{c_{i, 1}\}_{i \in \mathbb{\mathbb{Q}^{+}}}$ read backwards is an $\ind$-Morley sequence over $M$ starting with $c_{1, \infty}$. So will $\{c_{i, 1}\}_{i \in \mathbb{\mathbb{Q}^{+}}}$ will witness Kim-dividing of $\varphi(x, c_{1, \infty})$. Thus we only have to show the following claim, concluding by the generalized freedom axiom as in the other direction: 

\begin{claim}
There is an $\ind$-Morley sequence over $M$ of realizations of $\mathrm{tp}(c_{1, \infty}/M')$
\end{claim}

\begin{proof}
Exactly as in Claim \ref{1-indmorleysequenceexistence}.
\end{proof}

\end{proof}

\begin{example}
In Examples \ref{1-freeamalgamationtheories} and \ref{1-notfullystationary}, $\ind$-Kim-independence coincides with $\ind^{a}$. Clearly it implies $\ind^{a}$. Now suppose $a\ind^{a}_{M} b$. By extension for $\ind^{a}$ we can assume $a$ and $b$ are algebraically closed sets, or models, containing $M$. Now, in Example \ref{1-freeamalgamationtheories}, $a\ind^{K^{\ind}}_{M} b$ follows from Lemma 7.6 of \cite{Co15} (or the proof of Remark \ref{1-freedomgeneralizedfreedom}). In Example \ref{1-notfullystationary} (where $\ind$ differs from free amalgamation in general), it follows from the discussion in that example.

We will generalize this discussion to all free amalgamation theories in Proposition \ref{1-trivialconantindependence}.
\end{example}

\begin{example}
If $T$ is $\mathrm{NSOP}_{1}$, $\ind$-Kim independence will always coincide with Kim-independence. This is by Kim's lemma in $\mathrm{NSOP}_{1}$ theories, Fact \ref{1-kimslemma}.
\end{example}

Suppose $\ind$ satisfies the generalized freedom axiom and the equivalent conditions of Theorem \ref{1-relativensop1}. Then the (in this case superficially) relative notion of $\ind$-Kim independence is not really a relative notion at all, but rather a new notion of independence with an intrinsically model-theoretic definition.

\begin{definition}\label{1-conantindependence}
Let $M$ be a model and $\varphi(x, b)$ a formula. We say $\varphi(x, b)$ \emph{Conant-divides} over $M$ if for \emph{every} invariant Morley sequence $\{b_{i}\}_{i < \omega}$ over $M$ starting with $b$, $\{\varphi(x, b)\}_{i < \omega}$ is inconsistent. We say $\varphi(x, b)$ \emph{Conant-forks} over $M$ if and only if it implies a disjunction of formulas Conant-dividing over $M$. We say $a$ is \emph{Conant-independent} from $b$ over $M$, written $a \ind^{K^{*}}_{M}b$, if $\mathrm{tp}(a/Mb)$ does not contain any formulas Conant-forking over $M$.
\end{definition}

Note that Conant-dividing is just ``strong Kim-dividing," Definition 5.1 of \cite{KRS19}.

\begin{corollary}
Suppose $\ind$ satisfies the generalized freedom axiom. Then if $\ind$-Kim independence is symmetric, it coincides with Conant-independence.
\end{corollary}

\begin{proof}
By Theorem 
\ref{1-relativensop1}, $\ind$-Kim-dividing coincides with Conant-dividing, because $\ind$ satisfies the relative Kim's lemma. So $\ind$-Kim-forking coincides with Conant-forking.
\end{proof}

Conant-independence will coincide with Kim-independence in $\mathrm{NSOP}_{1}$ theories and with $\ind^{a}$ in free amalgamation theories (see below), so it is not readily apparent from these examples that Conant-independence is a new independence notion. Nonetheless, in section 5 we will discuss some interesting examples of strictly $\mathrm{NSOP}_{4}$ theories not covered by these cases.

Note that a related notion called ``Conant-independence" is defined using \textit{finitely satisfiable} Morley sequences in \cite{NSOP2}, where it is shown to coincide with Kim-independence in $\mathrm{NSOP}_{2}$ theories. Despite the fact that the choice between invariant or finitely satisfiable Morley sequences does not matter for Kim-independence in $\mathrm{NSOP}_{1}$ theories (see Fact 2.1), it is not known when our notion of Conant-independence and the one from \cite{NSOP2} coincide\footnote{We suspect that they do not, even in $\mathrm{NSOP}_{4}$ theories. If Conant-independence with respect to finitely satisfiable Morley sequences coincided with the standard Conant-independence $\ind^{a}$ in the triangle-free random graph, and there were also least finitely satisfiable types in the restriction of the Kim-dividing order to finitely satisfiable types, then $\ind^{a}$ would have to satisfy a ``weak independence theorem" (see Proposition 6.10 of \cite{KR17} for the original result, or proposition 5.7 of \cite{NSOP2}, whose proof is quoted below, for a result involving Conant-independence with respect to finitely satisfiable Morley sequences) with respect to those least finitely satisfiable types. But satisfying a weak independence theorem for $\ind^{a}$ can be seen in this example to characterize free amalgamation (the standard one, with no new edges). So if Definition \ref{1-conantindependence} coincides with the definition from \cite{NSOP2}, the invariant types given by standard free amalgamation in the triangle-free random graph (that is, the types $p^{*}$ as in Fact \ref{1-morleysequencebasics}.1, where $\ind$ is free amalgamation) are finitely satisfiable. We do not think that the invariant types given by free amalgamation in the triangle-free random graph are finitely satisfiable, suggesting that Definition \ref{1-conantindependence} does not coincide with the definition from \cite{NSOP2}}.

\begin{example}
In an $\mathrm{NSOP}_{4}$ theory, or even a theory with symmetric Conant-independence, the relation $\ind$ might satisfy the generalized freedom axiom but not the equivalent conditions of Theorem \ref{1-relativensop1}. Consider the theory of the generic $K_{3}$-free graph with two constants, $c$ and $d$, for distinct vertices not connected by an edge. (That $c$ and $d$ are not connected is not essential here, but it is included for completeness of $T$.) This is $\mathrm{NSOP}_{4}$, originally by work of Shelah (\cite{She95}). Declare $A \ind_{M} B$ if the following hold

(a) $A \cap B \subseteq M$, and

(b) a node $a$ of $A \backslash M$ and a node $b$ of $B \backslash M$ have an edge between them if and only if

(i) $a$ connects to the constant $c$, and $b$ connects to the constant $d$, and

(ii) $a$ and $b$ connect to no common vertices in $M$.

This makes sense as an abstract stationary relation satisfying the assumptions at the beginning of this section; we will show full existence. To see this, suppose $AM$ and $BM$ form triangle-free graphs with $A \cap B \subseteq M$, and that $A$ and $B$ are amalgamated according to (b)(i), (b)(ii). We show that $ABM$ has no triangles. Without loss of generality, a triangle $e, f, g$ will have $e \in A \backslash M$, $f \in B \backslash M$, and $g$ in either $B \backslash M$ or in $M$. But the first case is impossible; by (b)(i) and the fact $f$ and $g$ both connect to $e$, $f$ and $g$ both connect to $d$. Because $MB$ has no triangles, so in particular $d, f, g$ is not a triangle, $f$ and $g$ cannot have an edge. So $e, f, g$ is not a triangle. Nor may $e, f, g$ form a triangle in second case, where $g \in M$, and $e, f \in M \backslash A$, because then $e$ and $f$ would connect to the common vertex $g \in M$, so will not be connected to each other, by (b)(ii).

We show the generalized freedom axiom for $\ind$. Let $M \prec M'$ and suppose that there is an $\ind$-Morley sequence $\{A_{i}\}_{i < \omega}$ over $M$ starting with $A$, which is indiscernible over $M'$. We show that an $\ind$-Morley sequence over $M'$ starting with $A$ will also be an $\ind$-Morley sequence over $M$. We first note two consequences of the hypothesis that there is an $\ind$-Morley sequence $\{A_{i}\}_{i < \omega}$ over $M$ starting with $A$, which is indiscernible over $M'$. First, clearly $A \cap M' \subseteq M$. Second,

(*) for all $e, f \in A \backslash M=A \backslash M'$ (possibly with $e = f$), if $e$ connects to $c$, $f$ connects to $d$, and $e$ and $f$ connect to no common vertices in $M$, then $e$ and $f$ connect to no common vertices in $M'$.

To show this, let $e_{1}A_{1} \equiv_{M} e A$ (so $e_{1} \in A_{1}$). It suffices to show that $e_{1}$ and $f$ connect to no common vertices in $M'$ under the hypotheses on $e$ and $f$. Note that $e_{1}$ connects to $c$, $f$ connects to $d$, and $e_{1}$ and $f$ connect to no common vertices of $M$. Moreover, $e_{1} \ind_{M} f$. So by definition of $\ind$, $e_{1}$ and $f$ have an edge. So they cannot connect to a common vertex $m \in M'$, because then $e_{1}, m, f$ would form a triangle.

Now let $\{A^{i}\}_{i \in \omega}$ denote an $\ind$-Morley sequence over $M'$ starting with $A$. We must show that $\{A^{i}\}_{i \in \omega}$ is also an $\ind$-Morley sequence over $M$. Let $n < \omega$; we show that $A^{n} \ind_{M} A^{< n}$. Clearly $A^{n} \cap M' \subseteq M$ and $A^{n} \cap A^{<n} \subseteq M'$ implies $A^{n} \cap A^{<n} \subseteq M$. It remains to show that $A^{n}$ and $A^{< n}$ satisfy criterion (b) of $A^{n} \ind_{M} A^{< n}$. Let $e \in A^{n}$, and suppose without loss of generality that $f \in A^{n-1}$. Suppose first of all that $e$ connects to the constant $c$, $f$ connects to the constant $d$, and $e$ and $f$ connect to no common vertices of $M$. We show that $e$ and $f$ must be connected. Let $f' \in A^{n}$ be such that $f'A^{n}=fA^{n-1}$. Then $e$ and $f'$ connect to no common vertices of $M$, while $e$ is connected to $c$ and $f'$ is connected to $d$. So by (*) and $A^{n} \equiv_{M} A$, $e$ and $f'$ connect to no common vertices of $M'$, so $e$ and $f$ connect to no common vertices of $M'$. Again, $e$ connects to $c$, $f$ connects to $d$, and $e$ and $f$ connect to no common vertices of $M'$. So because $e \ind_{M'} f$ (i.e. because $A^{n} \ind_{M} A^{<n}$ ), $e$ and $f$ must be connected. Now suppose that $e$ and $f$ are connected: we must show that $e$ connects to the constant $c$, $b$ connects to $d$, and $e$ and $f$ connect to no common vertices of $M$. But $e \ind_{M'} f$, so $e$ connects to the constant $c$, $b$ connects to $d$, and $e$ and $f$ connect to no common vertices of $M'$, a fortiori $M$.

However, $\ind$ does not satisfy the relative Kim's lemma (so neither is relative Kim-independence symmetric). Consider distinct disconnected vertices $\overline{c}=\{c_{1}, c_{2}\}$ outside of $M$, $c_{1}$ connected only to $c$ and to no other vertices of $M$, $c_{2}$ connected only to $d$ and to no other vertices of $M$. Consider the formula $\varphi(x, \overline{c}) =:  xRc_{1} \wedge xRc_{2}$. Clearly this does not Kim-divide with respect to the standard free amalgamation given by adding no new edges. But it $\ind$-Kim-divides, as if $\overline{c}^{1}, \overline{c}^{2}$ begin an $\ind$-Morley sequence of copies of realizations of $\mathrm{tp}(\overline{c}/M)$, then $c^{2}_{1}$ and $c^{1}_{2}$ are related by an edge, making it impossible for some other vertex to connect to both of them. So because $\varphi(x, c)$ $\ind$-Kim-divides over $M$, but its dividing is not witnessed by Morley sequences over $M$ with respect to the standard free amalgamation (which are invariant Morley sequences over $M$), the relative Kim's lemma fails for $\ind$.

Therefore, we cannot get a witnessing lemma for stationary independence relations for $\mathrm{NSOP}_{4}$ theories of the kind we obtained for $\mathrm{NSOP}_{2}$ theories in Theorem 4.3 of \cite{NSOP2} in order to prove the $\mathrm{NSOP}_{1}$-$\mathrm{SOP}_{3}$ dichotomy (so in that case, $\mathrm{NSOP}_{1}$). Moreover, this example tells us that even though the equivalent conditions of Theorem \ref{1-relativensop1} imply that $\ind$-Kim independence is just Conant-independence, in the statement of Theorem \ref{1-relativensop1} we must still consider $\ind$-Kim-independence itself and not just Conant-independence. This is because two things can happen at once in a theory $T$:

(a) Conant-independence is symmetric, even by virtue of an independence relation $\ind^{1}$, with the generalized freedom axiom, to which the equivalent conditions of Theorem \ref{1-relativensop1} apply. For example, the triangle-free random graph $T$, as noted in Example \ref{1-freeamalgamationtheories}, has for $\ind^{1}$ the free amalgamation from \cite{Co15}.

(b) There is some relation $\ind^{2}$ with the generalized freedom axiom but without the relative Kim's lemma, so that $\ind^{2}$-Kim independence, which will not be Conant-independence, remains asymmetric. For example, in the triangle-free random graph $T$, there is for $\ind^{2}$ the relation constructed in this example.
\end{example}

We conclude this section by isolating two model-theoretic assumptions, related by the generalized freedom axiom as in Theorem \ref{1-relativensop1} and without any known $\mathrm{NSOP}_{4}$ counterexamples, which together imply that a theory must be either $\mathrm{NSOP}_{1}$ or $\mathrm{SOP}_{3}$, and either $\mathrm{TP}_{2}$ or simple.

\begin{definition}\label{1-strongwitnessingproperty}
We say a theory $T$ has the \textit{strong witnessing property} if, for every $M \prec \mathbb{M}$, there is a global type $\mathbb{P}(\mathbb{X})$ extending $\mathrm{tp}(\mathbb{M}/M)$ (i.e. a type with parameters in $\mathbb{M}$ and unboundedly many variables $\mathbb{X}$ corresponding to the singletons of $\mathbb{M}$) so that the restriction $\mathbb{P}(\mathbb{X})|_{x}$ to boundedly many variables $x \subseteq \mathbb{X}$ is least in the Kim-dividing order among $M$-invariant extensions of $(\mathbb{P}(\mathbb{X})|_{x})|_{M}$.
\end{definition}

For example, let $\ind$ an independence relation satisfying the assumptions at the beginning of this section. By compactness and the assumptions on $\ind$, for all $M$ there is a type $\mathbb{P}(\mathbb{X})$ so that, for $a \models (\mathbb{P}(\mathbb{X})|_{x})|_{B}$, $a \ind_{M} B$. By the definition of the relative Kim's lemma, if $\ind$ satisfies the relative Kim's lemma, then $\mathbb{P}(\mathbb{X})$ realizes the strong witnessing property for $M$. To visualize this, consider, in the triangle-free random graph, the type $\mathbb{P}(\mathbb{X}) =: \mathrm{tp}(\mathbb{M}_{1}/\mathbb{M})$, where 
$\mathbb{M} \succ M \prec \mathbb{M}_{1}$, $\mathbb{M}, \prec \overline{\mathbb{M}} \succ \mathbb{M}_{1}$ for some very large model $\overline{\mathbb{M}}$, $\mathbb{M}_{1} \equiv_{M} \mathbb{M}$, and $\mathbb{M}_{1}$ and $\mathbb{M}$ are freely amalgamated over $M$ as in Example \ref{1-freeamalgamationtheories}.

\begin{theorem}
If a theory $T$ satisfies the strong witnessing property and has symmetric Conant-independence, then it is either $\mathrm{NSOP}_{1}$ or $\mathrm{SOP}_{3}$.
\end{theorem}

\begin{proof}
If $p(x)$ is a type over $M$, then define a \textit{strong witnessing extension} of $p(x)$ to be a  global extension $q(x)$ of $p(x)$ with the following property: 

For all tuples $b \in \mathbb{M}$, if $c \in \mathbb{M}$ with $c \models q(x)|_{Mb}$, then for any $a \in \mathbb{M}$ there is $a' \equiv_{Mc}a$ so that $\mathrm{tp}(a'c/Mb)$ extends to an $M$-invariant type least in the Kim-dividing order.

This is just the definition of ``strong canonical coheir," but replacing ``canonical coheir" with ``$M$-invariant type least in the Kim-dividing order". By the strong witnessing property, strong witnessing types extending any type $p(x)$ over a model $M$ exist. The proof is formally the same as Lemma 5.5 of \cite{NSOP2}; instead of finding $\mathbb{M}_{1} \equiv_{M} \mathbb{M}$ with $\mathbb{M} \ind_{M}^{\mathrm{CK}} \mathbb{M}_{1}$, we find $\mathbb{M}_{1} \models \mathbb{P}(\mathbb{X})$. (In fact, existence of a strong witnessing types extending any small type $p(x)$ over a monster model can be shown to be equivalent to the strong witnessing property.)

Conant-dividing is the same as Kim-dividing witnessed by a Morley sequence in some (any) strong witnessing type. Therefore, Conant-forking is the same as Conant-dividing as in Proposition \ref{1-relativeforkingdividing}. Meanwhile, Conant-independence is symmetric by assumption, and the chain condition for Conant-independence with respect to Morley sequences in strong witnessing types is as in Claim \ref{1-cc}. So the result follows nearly word-for-word from the proofs of Proposition 5.7 of \cite{NSOP2} and the discussion in Section 6 in the same paper. We just replace any reference to the coheir notions of Conant-independence and Kim-dividing independence with the invariant notions, and replace any reference to strong canonical coheirs and canonical Morley sequences with strong witnessing types and invariant Morley sequences in strong witnessing types. This works because all we use about canonical Morley sequences in these arguments is that if a canonical Morley sequence over $M$ witnesses the Kim-dividing of a formula $\varphi(x, b)$ over $M$, then $\varphi(x, b)$ divides with respect to all coheir Morley sequences.

Note that the proof of Proposition 5.7 of \cite{NSOP2} comes directly from the proof of the ``weak independence theorem" (Proposition 6.10) of Kaplan and Ramsey in \cite{KR17}. It plays the role in our argument that the freedom axiom plays in Theorem 7.17 of \cite{Co15}, showing that a modular free amalgamation theory must be either simple or $\mathrm{SOP}_{3}$. The proof of that theorem serves as a basis for section 6 of \cite{NSOP2}, which requires a different argument in Claim 6.1 of that section; either the original one from \cite{NSOP2} or the alternative argument of \cite{Lee22} referenced in that paper.
\end{proof}

The following follows Theorem 7.7 of Conant (\cite{Co15}), using a similar argument:

\begin{theorem}\label{1-classification}
If a theory $T$ satisfies the strong witnessing property, then it is either $\mathrm{TP}_{2}$ or simple.
\end{theorem}

\begin{proof}
If $T$ is not simple then dividing does not coincide with Conant-dividing, because otherwise $T$ would satisfy Kim's lemma for dividing. So there is a model $M$ and a formula $\varphi(x,b)$ that Kim-divides but does not Conant-divide over $M$. By dividing of $\varphi(x, b)$ over $M$, there is an indiscernible sequence $I=\{b_{i}\}_{i < \kappa}$ for very large $\kappa$, starting with $b$, and $\varphi(x, b) \in \mathrm{tp}(a/Mb)$ so that $\{\varphi(x, b_{i})\}_{i \in I}$ is $k$-inconsistent for some $k$. It follows from the strong witnessing property that we can find some $M$-invariant global type $p(\overline{y})$ (where $\overline{y} = \{y_{i}\}_{i < \kappa}$) extending $\mathrm{tp}(I/M)$ with the following property: for $i < \omega$, $p(\overline{y})|_{y_{i}}$ is least in the Kim-dividing order among $M$-invariant extensions of $(p(\overline{y})|_{y_{i}})|_{M}=\mathrm{tp}(b/M)$

By the pigeonhole principle we can assume that $p(\overline{y})|_{y_{i}}$ is constant in $i < \kappa$, by restricting the original $p(\overline{y})$ to a sub-tuple of variables\footnote{A version of this style of argument is used in the proof of Lemma 3.12 of \cite{CK09}.}. More precisely, there is some global $M$-invariant type $q(y)$ so that for $i < \omega$, $p(\overline{y})|_{y_{i}}= q(y_{i})$. Now take an invariant Morley sequence $\{I_{i}\}_{i < \omega}$ over $M$  in $p(\overline{y})$. Let us take $I_{i}$ to be the rows of an array; since $I_{i} \equiv_{M} I$ for $i < \omega$, we see that the rows $I_{i}$ give us inconsistent sets of instances of $\varphi(x, b)$. To get $\mathrm{NTP}_{2}$, it remains to show the paths are consistent. But the paths are Morley sequences in $q(y)$, which is least in the Kim-dividing order among $M$-invariant extensions of $\mathrm{tp}(b/M)$. So since $\varphi(x, b)$ does not Conant-divide over $M$, the paths do not witness dividing of $\varphi(x, b)$, so give a consistent set of instances of that formula.
\end{proof}

\begin{corollary}
Suppose $\ind$ satisfies the generalized freedom axiom and the relative Kim's lemma. Then $T$ is either $\mathrm{NSOP}_{1}$ or $\mathrm{SOP}_{3}$, and is either simple or $\mathrm{TP}_{2}$.
\end{corollary}

\begin{remark}
Under the hypotheses of the corollary, we get a ``base monotone" version of this result (see section 2 of \cite{KR18} for related results on ``base monotone" versions of independence). Namely, if $M'\ind^{K^{*}}_{M} a$ and $M' \ind^{K^{*}}_{M} b$ and $a \ind_{M'} b $, then $a \ind_{M} b$. This is proven by applying the generalized freedom axiom rather than the chain condition in the proof of the ``weak independence theorem" analogue, and then applying stationarity. When in addition $\ind^{K^{*}}=\ind^{a}$, note the resemblance to the case of the freedom axiom where $C=M'$, $D = M$ are models and $C \cap AB \subseteq C \cap\mathrm{acl}(AD)\mathrm{acl}(BD)=D \subseteq C $.
\end{remark}

\section{Non-modular free amalgamation theories}

The following property of relations $\ind$ between sets follows from the ``full transitivity" from section 2:

\begin{definition}
The relation $\ind$ has \textit{base monotonicity} if $A \ind_B C$ and $B \subseteq D \subseteq C$ implies $A \ind_D C$.
\end{definition}

This is Proposition 8.8 of \cite{KR17}: 

\begin{fact}\label{1-basemonotone}
An $\mathrm{NSOP}_{1}$ theory is simple if and only if Kim-independence satisfies base monotonicity for $B = M \prec M' =D $ models.
\end{fact}

Conant asks (\cite{Co15}, Question 7.19) if any free amalgamation theory is \textit{modular}:

\begin{definition}
A theory is \emph{modular} if $\ind^{a}$ has base monotonicity. 
\end{definition}

Answering this question, we give an example of a nonmodular free amalgamation theory. Kruckman and Ramsey (\cite{KR18}) show that the empty theory in a language with a binary function symbol $f(x, y)$ has a strictly $\mathrm{NSOP}_{1}$ model completion, where Kim-independence coincides with $\ind^{a}$ and algebraic closure coincides with closure under $f(x, y)$. This theory is therefore non-modular. To show that forking coincides with dividing for complete types over models, they introduce a ``free amalgamation" relation that we expect to satisfy the generalized freedom axiom (see example \ref{1-strongindependence}). However, the ``free  amalgamation" from \cite{KR18} is not a free amalgamation relation in the sense of \cite{Co15}. We define a nonstandard relation $\ind$, which will satisfy the free amalgamation axioms. Let $T$ be the model completion of the empty theory in the language with a binary function symbol $f$, with an additional constant symbol $c$. Define $A \ind_{C} B$ to mean $A \cap B \subseteq C$ and, for $a \in A \backslash C$ and $b \in B \backslash C$, $f(b,a)=f(a, b) = c$. We show that $\ind$ is a free amalgamation relation. Invariance, monotonicity, and full transitivity are straightfoward. For full existence, we can enlarge $A, B$ to their algebraic closure with $C$, which we assume is algebraically closed. We can easily find a structure in the language extending $C$ where $A$ and $B$ embed disjointly over $C$, and where a point with coordinates properly in each of $A$ and $B$ will have image $c$. Full existence then follows from the fact that $T$ is the model completion. If $C \subseteq A \cap B$, $A, B, C$ algebraically closed, then $A \ind_{C} B$ determines the isomorphism type of $AB$ over $C$, so stationarity follows from quantifier elimination. For freedom, if $A \ind_{C} B$ and $C \cap AB \subseteq D \subseteq C$, then $A\cap B \subseteq C \cap AB \subseteq D$, while for $a \in A\backslash D \subseteq A \backslash C$, $b \in B\backslash D \subseteq B \backslash C$, $f(a,b) = f(b, a)=c$ as before. Finally, if $A \ind_{C} B$ for $C \subseteq A \cap B$ and $A, B, C$ algebraically closed, then the closure under $f(x, y)$, and therefore the algebraic closure, of $AB$ remains $AB$, yielding the closure axiom.

The existence of non-modular free amalgamation theories motivates the following generalization of Theorem 7.17 of \cite{Co15} that modular free amalgamation theories are either simple or $\mathrm{SOP}_{3}$:

\begin{theorem}\label{1-freeamalgamationnsop1sop3}
Free amalgamation theories are either $\mathrm{NSOP}_{1}$ or $\mathrm{SOP}_{3}$.
\end{theorem}

First, we observe that, justifying the terminology, Lemma 7.6 of Conant in \cite{Co15} is essentially a characterization of Conant-independence:

\begin{proposition}\label{1-trivialconantindependence}
Conant-independence in free amalgamation theories coincides with $\ind^{a}$ over models.
\end{proposition}

\begin{proof}
Clearly Conant-independence implies $\ind^{a}$. Conversely, suppose $a \ind^{a}_{M} b$. By extension for $\ind^{a}$, we can assume that $a$ and $b$ are algebraically closed sets containing $M$. But then by Lemma 7.6 of \cite{Co15}, there will be an $\ind$-Morley sequence over $M$ (extending Definition \ref{1-morleysequence} appropriately) starting with $b$ and indiscernible over $Ma$, so $a \ind^{K^{*}}_{C} b$.
\end{proof}

We also see that free amalgamation theories satisfy the strong witnessing property (Definition \ref{1-strongwitnessingproperty}). Namely for $M \prec \mathbb{M}$, let $\mathbb{P}(\mathbb{X})$ be the type extending $\mathrm{tp}(\mathbb{M}/M)$ so that for $x \subset \mathbb{X}$, $B \subset \mathbb{M}$ and $a \models (\mathbb{P}(\mathbb{X})|_{x})|_{MB}$, we have $\mathrm{acl}(a) \ind_{M} \mathbb{M}$. We can get this $\mathbb{P}(\mathbb{X})$ by compactness from stationarity and monotonicity of $\ind$. It suffices to show that, for $x \subseteq \mathbb{X}$, $\mathbb{P}(\mathbb{X})|_{x}$ is least in the Kim-dividing order among $M$-invariant extensions of $(\mathbb{P}(\mathbb{X})|_{x})|_{M}$. In other words, let $\varphi(x, b)$ be a formula not Conant-dividing over $M$, with $b \models (\mathbb{P}(\mathbb{X})|_{x})|_{M}$; we must show that invariant Morley sequences over $M$ in $\mathbb{P}(\mathbb{X})|_{x}$ starting with $b$ do not witness dividing of $\varphi(x, b)$. But invariant Morley sequences over $M$ in $\mathbb{P}(\mathbb{X})|_{x}$ starting with $b$ are just sequences $\{b_{i}\}_{i < \omega}$ with $b_{0} = b$ and $\mathrm{acl}(b_{i}) \ind_{M} \mathrm{acl}(b_{<i})$. So we just need to show that such a sequence does not witness dividing of $\varphi(x, b)$. But $\varphi(x, b)$, because it does not Conant-divide over $M$, must have a realization $a$ with $a \ind^{a}_{M}b$. So then we can proceed as in the above proof to show that sequences $\{b_{i}\}_{i < \omega}$ with $b_{0} = b$ and $\mathrm{acl}(b_{i}) \ind_{M} \mathrm{acl}(b_{<i})$ do not witness dividing of $\varphi(x, b)$.

Then, Theorem \ref{1-freeamalgamationnsop1sop3} follows from Theorem \ref{1-classification}.

If a free amalgamation theory $T$ is $\mathrm{NSOP}_{1}$, then by Fact \ref{1-kimslemma} and Proposition \ref{1-trivialconantindependence}, $\ind^{K}=\ind^{K^{*}}=\ind^{a}$. Thus the characterization of simple theories within the class $\mathrm{NSOP}_{1}$ in Fact \ref{1-basemonotone} gives us the following, extending Conant's result \cite{Co15} that a simple free amalgamation theory is modular:

\begin{proposition}
An $\mathrm{NSOP}_{1}$ free amalgamation theory is modular if and only if it is simple.
\end{proposition}

\section{Some examples}

We consider two examples of theories with relations satisfying the assumptions at the beginning of Section 3, as well as the generalized freedom axiom and the relative Kim's lemma. We characterize Conant-independence in these structures. Our purposes are twofold: to give a model-theoretic interpretation of certain tame independence relations in potentially strictly $\mathrm{NSOP}_{4}$ theories, and to extend the concept of free amalgamation to examples not covered by Conant's work in \cite{Co15}.

\begin{example}
(Finite-language countably categorical Hrushovski constructions.) We consider the case of the examples of $\omega$-categorical structure with a predimension introduced in section 3 of \cite{Ev02}, which is developed in \cite{EW09}. Let $\mathcal{L}$ be a language with finitely many relations (\cite{EW09} only require finitely many relations of each arity, but we include this requirement so that the predimension function only takes a discrete set of values), and for each relation symbol $R_{i}$, let $\alpha_{i}$ be a non-negative real number associated to $R_{i}$. For $A$ a finite $\mathcal{L}$-structure, define a \textit{predimension} $d_{0}(A)=|A| - \sum_{i} \alpha_{i} |R_{i}(A)|$, with $R_{i}(A)$ the set of tuples of $R_{i}$ with elements of $A$. Define the relation $A \leq B$ for $B$ any $\mathcal{L}$-structure and $A$ a finite substructure of $B$, to mean that every finite superstructure of $A$ within $B$ has predimension greater than $A$. Let $f$ be an increasing continuous positive real-valued function and let $\mathcal{C}_{f}$ be the class of finite $\mathcal{L}$-structures any substructure $A$ of which satifies $d_{0}(A) \geq f(A)$. Assume that, if $B_{1} \geq A \leq B_{2}$ belong to $\mathcal{C}_{f}$, then their evident ``free amalgamation," by taking their disjoint union over $A$ and adding no new edges, likewise belongs to $\mathcal{C}_{f}$. Then there is an $\mathcal{L}$-structure $M_{0}$ so that every finite substructure of $M_{0}$ belongs to $\mathcal{C}_{f}$, and so that if $B \geq A \leq M_{0}$ with $B$ finite, then there is an embedding $\iota: B \to M_{0}$ over $A$ so that $\iota(B) \leq M_{0}$. Let $T$ be the theory of $M_{0}$. The theory $T$ is $\omega$-categorical, so has bounded algebraic closure, and isomorphic algebraically closed sets  satisfy the same complete type. For $M$ a model of $T$, and $A \subseteq B \subseteq M$ with $A$ finite and $B$ any set, $A$ is algebraically closed in $B$ if and only if $A \leq B$, and $M$ will always continue to have the property that if $B \geq A \leq M$ with $B$ finite, then there is an embedding $\iota: B \to M$ over $A$ so that $\iota(B) \leq M$.

Though $T$ is not necessarily simple, \cite{EW09} show that it is either strictly $\mathrm{NSOP}_{4}$ or simple. However, it does have a natural notion of independence, even in the strictly $\mathrm{NSOP}_{4}$ case. This notion of independence, called $d$-independence, is defined in \cite{Ev02}; it will coincide with forking-independence in the simple case. For finite $A, B$, denote $d(A/B)=d_{0}(\mathrm{acl}(AB))-d_{0}(\mathrm{acl}(B))$ (recalling the bounded algebraic closure). This notion of relative dimension has a natural extension over infinite sets: for $A$ a finite set and $B$ any set, denote $d(A/B)=\min(\{d(A/B_{0}): B_{0} \subseteq B \: \mathrm{finite}\})$. We use the following notation for the relation referred to in \cite{Ev02} as \textit{$d$-independence}: for $a$ finite and $B, C$ any sets, let $a\ind^{d}_{B} C$ if and only if $d(a/BC)=d(a/B)$ and $\mathrm{acl}(aB) \cap \mathrm{acl}(CB)=\mathrm{acl}(B)$; for $a$, $B$, $C$ finite this last condition will be redundant. In \cite{Ev02} it is shown that $\ind^{d}$ has finite character and is symmetric, monotone and fully transitive where defined, so it extends naturally to a relation defined for $a$ possibly infinite with the same properties. We claim that there is a natural relation $\ind$ satisfying the assumptions at the beginning of section 3 as well as the generalized freedom axiom and the relative Kim's lemma, and that Conant-independence coincides with $\ind^{d}$ (so is in particular, symmetric).

We first observe a variant of property (P5) of \cite{Ev02} which, in place of a finitary analogue of the ``independence theorem" holding only in the simple examples, constitutes a base-monotone version of the ``weak independence theorem" with respect to free amalgamation analogous to those in \cite{KR17} with respect to coheir-independence or \cite{NSOP2} with respect to canonical coheirs. This is used implicitly in \cite{EW09} to show $\mathrm{NSOP}_{4}$, but we provide some justification.

(P5$'$) Let $B_{1} \geq A \leq B_{2}$ be finite algebraically closed sets such that $B_{1}$ and $B_{2}$ are \textit{freely amalgamated} over $A$, which is to say $\mathrm{acl}(B_{1}B_{2})$ is the disjoint union of $B_{1}$ and $B_{2}$ over $A$ with no new relations. Let $c_{1}$, $c_{2}$ be finite with $c_{1} \ind^{d}_{A} B_{1}$, $c_{1} \ind^{d}_{A} B_{2}$ and $c_{1} \equiv_{A} c_{2}$; then there is some $c$ realizing $\mathrm{tp}(c_{1}/B_{1}) \cup \mathrm{tp}(c_{2}/B_{2})$--with $c \ind^{d}_{A} B_{1}B_{2}$ (which is not needed here)--and $\mathrm{acl}(cB_{1})$ and $\mathrm{acl}(cB_{2})$ freely amalgamated over $\mathrm{acl}(cA)$.

When we only require that $B_{1} \ind^{d} B_{2}$ rather than that they be freely amalgamated, this is shown in Theorem 3.6(ii) of \cite{Ev02} under assumptions on $f$, so we need only observe that this proof works for this partial result without the assumptions on $f$. As in that proof we can form the $\mathcal{L}$-structure $F=E_{12} \cup E_{13} \cup E_{23}$ with no new relations, and with compatible isomorphisms $\varphi_{12}: \mathrm{acl}(B_{1}B_{2}) \to E_{12}$, $\varphi_{j3}: \mathrm{acl}(c_{j}B_{j}) \to E_{j3}$, which will be a special case of the construction from that proof where the ``underlying" predimension $y$ is just the cardinality. Now by point (i) of that proof, which does not use the additional assumption on $f$ required for simplicity, $E_{ij} \leq E$. The part of the proof where this additional assumption is required is point (ii), where it is shown that $F \in \mathcal{C}_{f}$; it must be shown that for each $D \subseteq F$, $d_{0}(D) \geq f(|D|)$. However, the assumption on $f$ is only used when $D$ is not contained in the union of two of the $E_{ij}$ (where the requirement follows by closure under free amalgamation). But $F = E_{13} \cup E_{23}$ because $\mathrm{acl}(B_{1}B_{2})=B_{1} \cup B_{2}$. So embedding a copy of $F$ over $B_{1}B_{2}$ (where $B_{1}B_{2}$ is identified by its image in $E_{12} \subset F$) so that it is algebraically closed will realize both types, and in a $d$-independent way by point (iii), which does not rely on the additional assumptions on $f$.

Now note that for $B_{1} \geq A \leq B_{2}$ algebraically closed finite sets and $c$ any finite set with $\mathrm{acl}(cB_{1})$ and $\mathrm{acl}(cB_{2})$ freely amalgamated over $\mathrm{acl}(cA)$, the type of $cB_{1}$ and $cB_{2}$ then completely determine the type of $cB_{1}B_{2}$ and in particular $B_{1} B_{2}$, so (P5$'$) implies that $B_{1}$ and $B_{2}$ are freely amalgamated over $A$. This observation leads to the following definition: for $M$ a model and $b, c$ finite sets of parameters, say $a \ind_{M} b$ if for any finite $A \leq M$ with $a \ind^{d}_{A} M$, $b \ind^{d}_{A} M$ (such an $A$ always exists because $d_{0}$ only takes a discrete set of values; see Lemma 2.17(a)(ii) of \cite{Ev02}), $\mathrm{acl}(aA)$ and $\mathrm{acl}(bA)$ are freely amalgamated over $A$. For existence, by compactness, it suffices to show that for types $p(x)$ and $q(y)$ over $M$, finitely many finite $A_{i} \leq M$ such that $p(x)$ and $q(x)$ $d$-independently extend their restrictions to $A_{i}$, and finite $B \subseteq M$, there are realizations $a$ of $p(x)|_{B}$ and $b$ of $q(y)|_{B}$ so that $\mathrm{acl}(aA_{i})$ and $\mathrm{acl}(bA_{i})$ are freely amalgamated over $A_{i}$ for each $i$. But take any $A \leq M$ containing each of the $A_{i}$ and $B$ and take realizations $a$ of $p(x)|_{A}$ and $b$ of $q(y)|_{A}$ so that $\mathrm{acl}(aA)$ and $\mathrm{acl}(bA)$ are freely amalgamated over $A$; then the free amalgamation conditions over the $A_{i}$, by the observation at the beginning of this paragraph, will be satisfied. By the quantifier elimination, $\ind$ is clearly stationary where it is defined. The relation $\ind$ is also monotone where defined, by the properties of free amalgamation for finite sets. So it extends to a relation $a\ind_{M} b$ for $a$, $b$ potentially infinite.

We next show that  $\ind^{K^{\ind}}$ implies $\ind^{d}$: The proof from \cite{Ev02} that dividing-independence implies $\ind^{d}$ (Lemma 2.19 (a) of \cite{Ev02}) cites Claim 1 of Theorem 4.2 of \cite{KP99}, whose proof will tell us than any $\ind^{d}$-independent sequence, such as $\ind$, witnesses dividing. For the convenience of the reader, we will recapitulate the argument. (Note that the argument of \cite{KP99} uses ``local character" which will be implicit in our argument.) Suppose $a \nind_{M}^{d} b$. We show that $a\nind_{M}^{K^{\ind}}b$; by finite character of $\ind^{d}$ we may assume $a$ is finite. To show $a\nind_{M}^{K^{\ind}}b$ it suffices by Proposition \ref{1-relativeforkingdividing} to show that, for $\{b_{i}\}_{i < \kappa}$ a long $\ind$-Morley sequence starting with $b$, there is no $a'$ so that $a'b_{i} \equiv_{M} ab$ for $i < \kappa$. Suppose such an $a'$ existed. The relation $\ind$, by definition, implies $\ind^{d}$. Because $\{b_{i}\}_{i < \kappa}$ is an $\ind$-Morley sequence over $M$, it has the property that $b_i \ind^{d}_{M} b_{<i}$ for $i < \kappa$. So by full transitivity and symmetry for $\ind^{d}$, $a' \nind_{Mb_{<i}}^{d} b_{i}$ for $i < \kappa$. Let $d_{i}=\mathrm{d}(a'/Mb_{< i})$; this is defined, because $a$ is finite. Then for $i < j < \kappa$, $d_i < d_j$, which is impossible as $\kappa$ is large. (Because we are assuming the finite language case, we could actually have used a sequence $\{b_{i}\}_{i < \omega}$ rather than a long sequence.)

We next show the generalized freedom axiom for $\ind$. That is, we show that if $M' \ind^{K^{\ind
}}_{M} a$ and $\{a_{i}\}_{i \in I}$ is an $\ind$-Morley sequence starting with $a$ over $M'$, then it is an $\ind$-Morley sequence over $M$. Because $M' \ind^{K^{\ind
}}_{M} a$, $M' \ind^{{d}
}_{M} a$, by the above paragraph; therefore, $M' \ind_{M}^{d} {a_{i}}$ for $i < \omega$. Moreover, $a_{i} \ind_{M'} a_{<i}$ implies $a_{i} \ind^{d}_{M'} a_{<i}$ for $i < \omega$. So by repeated applications of symmetry and full transitivity of $
\ind^{d}$, $M' \ind^{d}_{M} a_{<i}$ as well, for $i < \omega$. To show $\{a_{i}\}_{i \in I}$ is an $\ind$-Morley sequence over $M$, it therefore remains to show the following claim: if $b' \ind_{M'} c$, $M'\ind_{M}^{d} b$, and $M'\ind_{M}^{d} c$, then $b' \ind_{M} c'$. If $A \leq M$ is finite, with $b \ind^{d}_{A} M$ and $c \ind^{d}_{A} M$, then it follows from $M' \ind_{M}^{d} b$ and $M' \ind^{K^{\ind
}}_{M} c$ that $b \ind^{d}_{A} M'$ and $c \ind^{d}_{A} M'$ (transitivity and symmetry). By this observation, our claim that $b' \ind_{M'} c$, $M'\ind_{M}^{d} b$, and $M'\ind_{M}^{d} c$ imply $b' \ind_{M} c'$ follows from the definition of $\ind$. So we have the generalized freedom axiom for $\ind$. We can also carry out a similar proof for a set in place of $M'$ (which we can assume to be algebraically closed and contain $M$), so $\ind^{d}$ implies $\ind^{K^{\ind}}$.

So $\ind^{d}$ coincides with $\ind$-Kim independence, which is then symmetric, and $\ind$ satisfies the generalized freedom axiom, so $\ind$ satisfies the relative Kim's lemma and $\ind^{d}$ coincides with Conant-independence.

\end{example}

\begin{example}
(Random graphs without small cycles). Shelah introduces this example in Claim 2.8.5 of \cite{She95}. Let $n \geq 3$, and consider first the case where $n$ is even. Then the theory of graphs without cycles of length not exceeding $n$ has a model companion $T$, but it is not the model completion. The theory $T$ does have quantifier elimination, however, in the graph language expanded by the definable partial function symbols $F^{k}_{m}$, for $k \leq m$ not more than $\frac{n}{2}$, sending vertices $a$ and $b$ of distance $m$ to the $k$th vertex along the path between $a$ and $b$; note that (particularly to the even case) any two vertices in $T$ have a unique path of length at most $\frac{n}{2}$ between them. (We adopt the convention that paths cannot retrace themselves.) A set $A \subseteq \mathbb{M}$ is then algebraically closed $A$ contains every path in $\mathbb{M}$ of length $\frac{n}{2}$ between every two vertices of $A$. The type of an algebraically closed set is determined by its quantifier-free type. Shelah (\cite{She95}) shows this theory is $\mathrm{NSOP}_{4}$.

We define a stationary relation $\ind$ over models, as follows. Define $a \ind_{M} b$ if $a \ind^{a}_{M} b$, there are no edges between $\mathrm{acl}(Ma)\backslash M$ and $\mathrm{acl}(Mb) \backslash M$, and the algebraic closure of $Mab$ is isomorphic over $\mathrm{acl}(Ma) \cup \mathrm{acl}(Mb)$ to the construction $P^{\infty}(\mathrm{acl}(Ma) \cup \mathrm{acl}(Mb))$ defined below. Note that this construction is free of any choice and thus gives, by quantifier elimination, a relation with invariance and stationarity. For $A$ a graph without cycles of length $\leq n$ let $P(A)$ be the disjoint union $A \sqcup A'=\sqcup_{d_{A}(a, b) > \frac{n}{2}} A'_{ab}$, where $A_{ab}'$ consists of new formal vertices $(a, b, 1), \ldots, (a, b, \frac{n}{2}-1)$ forming a path of length $\frac{n}{2}$ between $a$ and $b$, and the condition $d_{A}(a, b) > \frac{n}{2}$ denotes that $a$ and $b$ are vertices of $A$ admitting no path within $A$ of length $\leq \frac{n}{2}$. The edges of $P(A)$ will consist of those originally in $A$, as well as those making each $A_{ab}$ into the intermediate vertices of a path of length $\frac{n}{2}$ between $a$ and $b$. Note that $P(A)$, like $A$, will continue not to have cycles of length $\leq n$. This is because any cycle $C$ of length $\leq n$ cannot lie inside of $A$, so $A_{ab} \subseteq C$ for some $a, b \in A$ with no paths in $A$ of length $\leq \frac{n}{2}$. Then $A_{ab} \backslash C$ cannot lie in $A$, because otherwise $A_{ab} \backslash C$ would form a path in $A$ between $a$ and $b$ of length $\leq \frac{n}{2}$. So some other $A_{cd}$ must also lie in $C$, contradicting that $|C| \leq n$. Now define inductively $P^{0}(A) = A$, $P^{n}(A)=P(P^{n-1}(A))$, $P^{\infty}(A) = \cup^{\infty}_{n=0} P^{n}(A)$; we have defined $P^{\infty}(Mab)$ used in the definition of $\ind$. Since $T$ is the model companion, $\ind$ has existence, and we have to show monotonicity of $\ind$.

We first show two claims about our construction. For graphs $A \subseteq B$ without cycles of length $\leq n$, call $A$ \textit{closed} in $B$ if any path of length $\leq \frac{n}{2}$ between two vertices of $A$ lies in $B$. Within $\mathbb{M}$ this coincides with $A$ being relatively algebraically closed in $B$. Assume $A \subseteq B$ is closed in $B$. We claim:

(i) Define $P|_{A}(B)= A \sqcup \bigsqcup_{a, b \in A, d_{B}(a, b) > \frac{n}{2}} A_{ab}$ to be the union of $A$ with the new formal paths of $P(B)$ added between vertices of $A$. (Here $d_{B}(a, b) > \frac{n}{2}$ denotes that there is no path within $B$ between $a$ and $b$ of length $\leq \frac{n}{2}$.) Then $P|_{A}(B) = P(A)$.

(ii) $P|_{A}(B) = P(A)$ is closed in $P(B)$.

For (i), the only reason $P|_{A}(B) = P(A)$ can fail is if there are two vertices $a, b \in A$ so that there is no path of length $\leq \frac{n}{2}$ between $a$ and $b$ in $A$, but there is a path of length $\leq \frac{n}{2}$ between $a$ and $b$ in $B$. This contradicts the assumption that $A$ is closed in $B$. For (ii), let $\gamma$ be a path of length $\leq \frac{n}{2}$ in $P(A)$ between vertices $a, b \in P(A)$; we must show $\gamma \subseteq B$. The path $\gamma$ cannot contain one of the new formal paths $\{c, d\} \cup B_{cd}$ of $P(B)$ between $c, d$ not both in $A$, because then $\gamma$ would properly contain $\{c, d\} \cup B_{cd}$, and $\gamma$ would have length $> \frac{n}{2}$. The path $\gamma$ also cannot contain some new points of a new formal path $ B_{cd}$ of $P(B)$ where $c, d$ are not both in $A$, but not all of $\{c, d\} \cup B_{cd}$, because then $\gamma$ would start or end in $B_{cd}$ rather than starting or ending at $a$ or $b$. So every point of $\gamma$ lies either in the formal paths $A_{ab}=B_{ab}$ between $a, b \in A$, so lies in $P(A)$, or belongs to $B$. It follows that $\gamma$ can be decomposed into subpaths in $P(A)$ and subpaths in $B$ between two vertices of $A$. But any path in $B$ of length $\leq \frac{n}{2}$ between two vertices of $A$, will be contained in $A$, because $A$ is closed in $B$.

For graphs $A \subseteq B$ without cycles of length $\leq \frac{n}{2}$, define the \textit{closure} of $A$ in $B$ to be the smallest closed subgraph of $B$ containing $A$; in $\mathbb{M}$ this coincides with the algebraic closure. By induction, the following follows from (i) and (ii):

(*) For $A \subseteq B$, if $A$ is closed in $B$, then $P^{\infty}(A)$ is the closure of $A$ in $P^{\infty}(B)$.

We now show monotonicity of $\ind$. Assume $A \ind_{M} B$, and $A' \subseteq A$, $B' \subseteq B$; we show $A' \ind_{M} B'$. We show that $\mathrm{acl}(MA') \cup \mathrm{acl}(MB')$ is algebraically closed in $\mathrm{acl}(MA) \cup \mathrm{acl}(MB)$. Because $A \ind_{M} B$, there are no edges between $\mathrm{acl}(AM) \backslash M$ and $\mathrm{acl}(BM) \backslash M$. Therefore, any path $\gamma$ in $\mathrm{acl}(MA) \cup \mathrm{acl}(MB)$ with endpoints in $\mathrm{acl}(MA') \cup \mathrm{acl}(MB')$ can be decomposed into subpaths in $\mathrm{acl}(MA)$ with endpoints in $\mathrm{acl}(MA')$ and subpaths in $\mathrm{acl}(MB)$ with endpoints in   $\mathrm{acl}(MB')$. If $\gamma$ is a path of length $\leq \frac{n}{2}$, then these subpaths must be contained in $\mathrm{acl}(MA')$ and $\mathrm{acl}(MB')$, respectively. So $\mathrm{acl}(MA') \cup \mathrm{acl}(MB')$ is closed in $\mathrm{acl}(MA) \cup \mathrm{acl}(MB)$. Now it clearly follows from $A \ind_{M} B$ that $A' \ind_{M}^{a} B'$ and there are no edges between $\mathrm{acl}(MA') \backslash$ and $\mathrm{acl}(MB') \backslash M$. The rest follows from $A \ind_{M} B$, and the observation that $\mathrm{acl}(MA') \cup \mathrm{acl}(MB')$ is closed in $\mathrm{acl}(MA) \cup \mathrm{acl}(MB)$, by (*). 

Now define the relation $A\ind_{M}^{\frac{n}{4}}B$ to mean that $A \ind^{a}_{M} B$, and for $a, b \in \mathrm{acl}(AM) \cup \mathrm{acl}(BM)$, if there is no path of length at most $\frac{n}{4}$ between $a$ and $b$ in the graph $\mathrm{acl}(AM) \cup \mathrm{acl}(BM)$ with no new edges, then $a$ and $b$ are of distance greater than $\frac{n}{4}$ apart. (Note the importance of this distance restriction in Shelah's proof of $\mathrm{NSOP}_{4}$.) We claim that if $M'\ind^{\frac{n}{4}}_{M} a$, then an $\ind$-Morley sequence $\{a_{i}\}$ over $M'$ with $a_{i} \equiv_{M} a$ for $i < \omega$ will remain an $\ind$-Morley sequence over $M$. Clearly, $a_{i} \ind^{a}_{M} a_{<i}$ and there are no edges between $\mathrm{acl}(Ma_{i}) \backslash M$ and $\mathrm{acl}(Ma_{<i}) \backslash M$ for $i < \omega$. We will show that, if $M' \ind_{M}^{\frac{n}{4}} b$, $M' \ind^{\frac{n}{4}}_{M}c$, and $b \ind_{M'} c$, then $M' \ind_{M}^{\frac{n}{4}} bc$, and 

$$\mathrm{acl}(\mathrm{acl}(Mb) \cup \mathrm{acl}(Mc))=P^{\infty}(\mathrm{acl}(Mb) \cup \mathrm{acl}(Mc))$$

$$= P^{\infty}|_{\mathrm{acl}(Mb) \cup \mathrm{acl}(Mc)}(\mathrm{acl}(M'b)) \cup \mathrm{acl}(M'c)) \subseteq \mathrm{acl}(M'bc) = P^{\infty}(\mathrm{acl}(M'b)\cup \mathrm{acl}(M'c))$$

Then it will follow by induction that $\{a_{i}\}_{i < \omega}$ is an $\ind$-Morley sequence over $M$. We claim that $\mathrm{acl}(Mb) \cup \mathrm{acl}(Mc)$ is closed in $\mathrm{acl}(M'b)\cup \mathrm{acl}(M'c)$. Otherwise, some path $\gamma$ of length at most $\frac{n}{2}$ between $\mathrm{acl}(Mb)$ and $\mathrm{acl}(Mc)$ would have to pass through $M' \backslash M$. But $\gamma$ would then be too short not to pass between either $\mathrm{acl}(Mb)$ or $\mathrm{acl}(Mc)$, and $M' \backslash M$, in no greater than $\frac{n}{4}$ steps. Therefore, $\mathrm{acl}(Mb) \cup \mathrm{acl}(Mc)$ is closed in $\mathrm{acl}(M'b)\cup \mathrm{acl}(M'c)$. So by (*) above, $$\mathrm{acl}(\mathrm{acl}(Mb) \cup \mathrm{acl}(Mc))=P^{\infty}(\mathrm{acl}(Mb) \cup \mathrm{acl}(Mc))$$

It remains to show that $M' \ind_{M}^{\frac{n}{4}} bc$. But it can easily be seen that, for 

$$p \in \mathrm{acl}(\mathrm{acl}(Mb) \cup \mathrm{acl}(Mc)) \backslash \mathrm{acl}(Mb) \cup \mathrm{acl}(Mc)=P^{\infty}(\mathrm{acl}(Mb) \cup \mathrm{acl}(Mc)) \backslash \mathrm{acl}(Mb) \cup \mathrm{acl}(Mc) $$

the closest points of $\mathrm{acl}(M'b) \cup \mathrm{acl}(M'c)$ to $p$ within $\mathrm{acl}(M'bc) = P^{\infty}(\mathrm{acl}(M'b)\cup \mathrm{acl}(M'c))$ must be points of $\mathrm{acl}(Mb) \cup \mathrm{acl}(Mc)$. This is by construction of $P^{\infty}(\mathrm{acl}(M'b)\cup \mathrm{acl}(M'c))$. It follows that $M' \ind_{M}^{\frac{n}{4}} bc$.

This shows that $\mathrm{acl}(Mb) \cup \mathrm{acl}(Mc)$. Because the same reasoning works for a set (which can be assumed to be algebraically closed) in place of $M'$, we see that $\ind^{\frac{n}{4}}$ implies $\ind^{K^{\ind}}$.

We show the reverse implication, and thus that $\ind^{K^{\ind}} = \ind^{\frac{n}{4}}$. By the above discussion, this will tell us additionally that $\ind$ satisfies the generalized freedom axiom. By Theorem \ref{1-relativensop1}, it will follow that $\ind$ satisfies the relative Kim's lemma and $\ind^{\frac{n}{4}}$ is Conant-independence. (Of course, this last claim will also follow from the proof below, which shows directly that $\nind^{\frac{n}{4}}$ implies $\nind^{K^{*}}$, and the above fact that $\ind^{\frac{n}{4}}$ implies $\ind^{K^{\ind}}$ and therefore implies $\ind^{K^{*}}$). Suppose $A\ind^{\frac{n}{4}}_{M} B$ is false, and $A \ind^{a}_{M} B$. Then there is a path of length at most $\frac{n}{4}$ not passing through $M$ between a vertex $a$ of $\mathrm{acl}(MA)$ and a vertex $b$ of $\mathrm{acl}(MB)$. Let $\varphi(x, B) \in \mathrm{tp}(a/MB)$ imply that there is such a path. (Note that for a path of length at most $\frac{n}{4}$ not to pass through $M$, it need only avoid the finitely many elements of $M$ within distance $\frac{n}{4}$ of $b$.) Suppose $\{B_{i}\}_{i < \omega}$ is an invariant Morley sequence over $M$ (such as an $\ind$-Morley sequence over $M$) starting with $B$, so that $\{\varphi(x,B_{i})\}_{i < \omega}$ is consistent, realized by some $A'$. This realization $A' \models \{\varphi(x,B_{i})\}_{i < \omega}$ can be chosen so that $\{B_{i}\}_{i < \omega}$ is indiscernible over $MA$, so over $\mathrm{acl}(MA')$. Then $A'$ will have a vertex $a'$ lying on a path of length at most $\frac{n}{2}$ between vertices of $\mathrm{acl}(Mb_{0}) \cup \mathrm{acl}(Mb_{1})$ avoiding $M$. Therefore, $a' \in \mathrm{acl}(Mb_{0}b_{1})\backslash M$. Similarly, $a' \in \mathrm{acl}(MB_{2}B_{3})\backslash M$. But the concatenation $\{B_{2i}B_{2i+1}\}_{i < \omega}$ remains an invariant Morley sequence, so $B_{0}B_{1} \ind_{M}^{a} B_{2}B_{3}$, a contradiction.

Note that $\ind^{\frac{n}{4}}$ does not coincide with $\ind^{a}$, making this another interesting case of Conant-independence. To see this, consider a vertex $a$ of distance $\frac{n}{2}$ from the model $M$ and take some algebraically closed graph (that is, a graph with no two vertices farther than $\frac{n}{2}$ apart) $B \supset M$ containing $M$ and $a$, then take two disjoint copies $B_{1}$ and $B_{2}$ of this graph over $M$, with $a_{1}$ and $a_{2}$ the copies of $a$ over $M$, and no further edges. Then we can add a path of length at most $\frac{n}{4}$ between $a_{1}$ and $a_{2}$ and not create any small cycles. Embed this into a larger model over $M$, and the images of $B_{1}$ and $B_{2}$ will be independent according to $\ind^{a}$ but not $\ind^{\frac{n}{2}
}$.

The case where $n=2m+1$ is odd is different in that, while the quantifier elimination still holds in the language expended by the definable partial function symbols, two vertices can be of length $m+1$ apart and none of the partial function symbols can be defined there, in which case there are infinitely many paths of length $m+1$ between them. So defining $\ind$ is easier: let $A \ind_{M} B$ if $A \ind^{a}_{M} B$, and for two vertices $a, b \in \mathrm{acl}(MA) \cup \mathrm{acl}(MB)$ that are not already of distance at most $m$ apart within $\mathrm{acl}(Ma) \cup \mathrm{acl}(Mb)$ with no new edges, $a$ and $b$ will have distance $m+1$. Then a similar analysis holds.

\end{example}

\section{Conant-independence in the $\mathrm{NSOP}_{n}$ hierarchy}

We prove that symmetry of Conant-independence implies $\mathrm{NSOP}_{4}$. We begin with the following fact, whose proof is essentially that of Proposition 5.2 of \cite{NSOP2}:

\begin{fact}\label{1-coheirs}
If formula $\varphi(x, b)$ Conant-forks over $M$, then if $\{b_{i}\}_{i < \omega}$ is an invariant Morley sequence over $M$ in an $M$-finitely satisfiaable type with $b_{0} = b$, $\{\varphi(x, b_{i})\}_{i < \omega}$ is inconsistent.
\end{fact}

\begin{proof}

Let $\models \varphi(x, b)\rightarrow \bigvee_{i = 1}^{n} \psi_{i}(x, c_{i})$ for $\psi_{i}(x, c_{i})$ Conant-dividing over $M$, so in particular Kim-dividing over $M$ by any invariant Morley sequence in an $M$-finitely satisfiable type extending $\mathrm{tp}(c_{i}/M)$. By left extension and monotonicity for $\ind^{u}$ , whether or not a formula $\ind$-Kim-divides over $M$ does not change when adding unused parameters, so we can assume $c_{i} = b$ for $1 \leq i \leq n$. Then $\varphi(x, b)$ Kim-divides over $M$ by any invariant Morley sequence in a finitely satisfiable type over $M$, for suppose otherwise. Let $\{b_{i}\}_{i < \omega}$ be an invariant Morley sequence in a finitely satisfiable type over $M$ starting with $b$; then there will be some $a$ realizing $\{\varphi(x, b_{i})\}_{i < \omega}$. So by the pigeonhole principle, there will be some $1 \leq j \leq n$ so that $a$ realizes $\{\psi_{j}(x, b_{i})\}_{i \in S}$ for $S \subseteq \omega$ infinite. But by monotonicity and an automorphism, we can assume $\{b_{i}\}_{i \in S}$ is an invariant Morley sequence in that same finitely satisfiable type over $M$ starting with $b$, contradicting Conant-dividing of $\psi_{j}(x, b)$.
\end{proof}

The following uses similar Skolemization methods to Proposition 5.6 of Chernikov and Ramsey in \cite{CR15}, which generalize in a surprising way to indiscernible sequences ordered by a definable relation with no $4$-cycles.

\begin{theorem}
Any theory where Conant-forking is symmetric is $\mathrm{NSOP}_{4}$. Thus $n=4$ is the greatest $n$ so that there are strictly $\mathrm{NSOP}_{n}$ theories with symmetric Conant-independence.
\end{theorem}

\begin{proof}
Suppose a theory $T$ has $\mathrm{SOP}_{4}$; we show that Conant-independence cannot be symmetric. Let $R(x, y)$ be a definable binary relation with no $4$-cycles, such that there exists an infinite sequence $\langle a_{i}\rangle_{i \in I}$ so that $R(a_{i}, a_{j})$ for $i < j$. Fixing a Skolemization of $T$, we can assume that this sequence is indiscernible in that Skolemization and is of the form $\langle c_{i}\rangle_{i < \omega} + \langle a_{1} \rangle + \langle b_{1} \rangle + \langle a_{2} \rangle + \langle b_{2} \rangle + \langle a_{3} \rangle + \langle c_{i}\rangle_{i < \omega^{*}}$. Let $M = \mathrm{dcl}_{\mathrm{Sk}}(\langle c_{i}\rangle_{i < \omega} + \langle c_{i}\rangle_{i < \omega^{*}})$, $\overline{a} = a_{1}a_{2}a_{3}$, $\overline{b}=b_{1}b_{2}$; we show $\overline{a} \ind^{K^{*}}_{M} \overline{b}$ but $\overline{b} \nind^{K^{*}}_{M} \overline{a}$. For the first part, clearly $\langle c_{i}\rangle_{i < \omega} + \langle a_{1} \rangle + \langle b_{1} \rangle + \langle a_{2} \rangle + \langle b_{2} \rangle + \langle a_{3} \rangle + \langle c_{i}\rangle_{i < \omega^{*}}$ is contained in a sequence, indiscernible in the Skolemization, of the form $\langle c_{i}\rangle_{i < \omega} + \langle a_{1} \rangle + \langle b^{i}_{1} \rangle_{i < \omega} + \langle a_{2} \rangle + \langle b_{2}^{i} \rangle_{i < \omega^{*}} + \langle a_{3} \rangle + \langle c_{i}\rangle_{i < \omega^{*}}$, with $b_{j}^{0} = b_{j}$ for $j = 1, 2$. But $\langle b_{1}^{i}b_{2}^{i} \rangle_{i < \omega}$ is a coheir Morley sequence over $M$ starting with $\overline{b}$ and indiscernible over $M\overline{a}$, so by Fact \ref{1-coheirs} we get $\overline{a} \ind^{K^{*}}_{M} \overline{b}$. For the dependent direction, we show $R(a_{1}, y_{1}) \wedge R(y_{1}, a_{2}) \wedge R(a_{2}, y_{2}) \wedge R(y_{2}, a_{3}) \in \mathrm{tp}(\overline{b}/M\overline{a})$ Conant-divides over $M$. Let $\langle a^{i}_{1} a^{i}_{2} a^{i}_{3}\rangle_{i < \omega}$ be an $M$-invariant Morley sequence starting with $\overline{a}$ and suppose $\{R(a^{i}_{1}, y_{1}) \wedge R(y_{1}, a^{i}_{2}) \wedge R(a^{i}_{2}, y_{2}) \wedge R(y_{2}, a^{i}_{3})\}_{i < \omega}$ were consistent, realized by $b'_{1}b'_{2}$. Then $\models R(a^{1}_{2}, b'_{2}) \wedge R(b'_{2}, a^{0}_{3}) $. Now $\models \exists x R(a^{0}_{1}, x) \wedge R(x, a_{2}^{1}) $, witnessed by $b_{1}'$. But $a_{1}^{0} \equiv_{M} a_{3}^{0}$, so by invariance, $a_{1}^{0} \equiv_{M\overline{a}^{1}} a_{3}^{0}$, and in particular $a_{1}^{0} \equiv_{Ma_{2}^{1}} a_{3}^{0}$. So $\models \exists x R(a^{0}_{3}, x) \wedge R(x, a_{2}^{1}) $, witnessed, say, by $b_{1}''$. But $\models R(a^{1}_{2}, b'_{2}) \wedge R(b'_{2}, a^{0}_{3}) \wedge R(a_{3}^{0}, b''_{1}) \wedge R(b''_{1}, a_{2}^{1}) $, a $4$-cycle, contradiction.
\end{proof}

Thus one of the three main classification-theoretic properties Conant proved for free amalgamation theories in \cite{Co15}--they are either $\mathrm{NSOP}_{1}$ or $\mathrm{SOP}_{3}$, are either simple or $\mathrm{TP}_{2}$, and are $\mathrm{NSOP}_{4}$--holds solely under the assumpton of symmetric Conant-independence. So far, we are only able to prove the other two identities with an additional assumption about invariant types least in the Kim-dividing order (see section $3$; note that Theorem \ref{1-classification} in fact says that this additional assumption about the Kim-dividing order, and not symmetric Conant-independence, is needed to prove $\mathrm{TP}_{2}$ or simple). These assumptions, symmetry for $\ind^{K^{*}}$ and the strong witnessing property, together generalize the free amalgamation theories, and they are related by the generalized freedom axiom using Theorem \ref{1-relativensop1}. Neither assumption is known to fail in any $\mathrm{NSOP}_{4}$ theory. However, it would be desirable if we had a criterion analogous to the theory of independence in simple or $\mathrm{NSOP}_{1}$ theories that gave us all of the classification-theoretic properties of free amalgamation theories. Can we get those other two properties with just symmetry for Conant-independence alone?

\begin{problem}
Must a theory with symmetric Conant-independence be either simple or $\mathrm{TP}_{2}$? Must it be either $\mathrm{NSOP}_{1}$ or $\mathrm{SOP}_{3}$? 
\end{problem}

We are also interested in extending the theory of Kim-independence beyond $\mathrm{NSOP}_{1}$. Given that the class of strictly $\mathrm{NSOP}_{4}$ theories is the most complicated classification-theoretic class where Conant-independence is symmetric, we may ask whether symmetry for Conant-independence characterizes $\mathrm{NSOP}_{4}$ the same way symmetry for Kim-independence characterizes $\mathrm{NSOP}_{1}$.

\begin{problem}
In an $\mathrm{NSOP}_{4}$ theory, is Conant-independence always symmetric?
\end{problem}

A positive answer to both the last problem and one of the two questions from the previous problem will solve some of the open regions of the classification-theoretic hierarchy, further underscoring the connections between classification theory and the theory of model-theoretic independence.

\bibliographystyle{plain}
\bibliography{refs}

\end{document}